\documentclass[10pt,reqno]{amsart}
\usepackage[utf8]{inputenc}
\usepackage{amsopn,amssymb,mathrsfs,xcolor,bm,url,enumitem}
\usepackage[abbrev]{amsrefs} 
\usepackage[a4paper,lmargin=3cm,rmargin=3cm,tmargin=4cm,bmargin=4cm,marginparwidth=2.8cm,marginparsep=1mm]{geometry} 
\usepackage[bookmarks=true,hyperindex,pdftex,colorlinks,citecolor=blue]{hyperref} 

\newtheorem{theorem}{Theorem}[section]
\newtheorem{corollary}[theorem]{Corollary}
\newtheorem{lemma}[theorem]{Lemma}
\newtheorem{proposition}[theorem]{Proposition}
\newtheorem{problem}[theorem]{Problem}

\theoremstyle{definition}
\newtheorem{definition}[theorem]{Definition}
\newtheorem{example}[theorem]{Example}
\newtheorem{remark}[theorem]{Remark}

\newcommand{\ds}{\displaystyle}
\newcommand{\ep}{\varepsilon}
\newcommand{\fnii}[2]{\ensuremath{{#1}^{-1}[{#2}]}}
\newcommand{\ind}[1]{\ensuremath{\mbox{\boldmath{$1$}}_{#1}}}
\newcommand{\lint}[4]{\ensuremath{\int_{#1}^{#2}{#3}\:\mathrm{d}{#4}}}
\newcommand{\map}[3]{\ensuremath{{#1}:{#2}\to{#3}}}
\newcommand{\md}{\ensuremath{\mathrm{d}}}
\newcommand{\me}{\ensuremath{\mathrm{e}}}
\newcommand{\n}[1]{\ensuremath{\left\|{#1}\right\|}}
\newcommand{\N}{\mathbb{N}}
\newcommand{\ndot}{\ensuremath{\left\|\cdot\right\|}}
\newcommand{\pn}[2]{\ensuremath{\left\|{#1}\right\|_{#2}}}
\newcommand{\pndot}[1]{\ensuremath{\left\|\cdot\right\|_{#1}}}
\newcommand{\Q}{\mathbb{Q}}
\newcommand{\R}{\mathbb{R}}
\newcommand{\restrict}[1]{\ensuremath{\mathord{\upharpoonright}_{#1}}}
\newcommand{\set}[2]{\ensuremath{\left\{{#1}\;:\;\,{#2}\right\}}}
\newcommand{\ts}{\textstyle}
\newcommand{\Z}{\mathbb{Z}}

\DeclareMathOperator{\Lip}{Lip}
\DeclareMathOperator{\sgn}{sgn}
\DeclareMathOperator{\supp}{supp}

\numberwithin{equation}{section}

\allowdisplaybreaks

\title{Convergence Analysis of the Geometric Thin-Film Equation}

\author[L. \'O N\'araigh]{Lennon \'O N\'araigh}
\address[L. \'O N\'araigh]{School of Mathematics and Statistics, University College Dublin, Belfield, Dublin 4, Ireland}
\email{onaraigh@maths.ucd.ie}

\author[K. E. Pang]{Khang Ee Pang}
\address[K. E. Pang]{School of Mathematics and Statistics, University College Dublin, Belfield, Dublin 4, Ireland}
\email{khang-ee.pang@ucdconnect.ie}

\author[R. J. Smith]{Richard J. Smith}
\address[R. J. Smith]{School of Mathematics and Statistics, University College Dublin, Belfield, Dublin 4, Ireland}
\email{richard.smith@maths.ucd.ie}

\date{\today}

\begin{document}

\begin{abstract}
The Geometric Thin-Film equation is a mathematical model of droplet spreading in the long-wave limit, which includes a regularization of the contact-line singularity. We show that the weak formulation of the problem, given initial Radon data, admits solutions that are globally defined for all time and are expressible as push-forwards of Borel measurable functions whose behaviour is governed by a set of ordinary differential equations (ODEs). The existence is first demonstrated in the special case of a finite weighted sum of delta functions whose centres evolve over time -- these are known as `particle solutions'. In the general case, we construct a convergent sequence of particle solutions whose limit yields a solution of the above form. Moreover, we demonstrate that all weak solutions constructed in this way are $1/2$-H\"older continuous in time and are uniquely determined by the initial conditions.
\end{abstract}

\keywords{Droplet spreading, thin-film equation, Geometric Thin-Film Equation, particle solutions, peakon push-forward solutions}
\subjclass{Primary 76A20, 35D30; Secondary 46N20}
\maketitle

\section{Introduction}

When a fluid rests on a substrate, the free-surface height is a key variable which characterizes the fluid mechanics.  In the longwave limit, where the horizontal length scale of the fluid is much larger than the vertical length scale, the Navier-Stokes equations reduce to a fourth-order nonlinear pa\-ra\-bo\-lic equation for the free-surface height $h(x,t)$, this is the thin-film equation \cite{oron1997} (in this work, we have only one spatial variable $x$).   In the simplest setting where the thin film is deposited on a flat non-inclined surface without external forces and only driven by surface tension, the evolution of  $h(x,t)$ satisfies the thin-film equation
\begin{equation}
\partial_t h = -\partial_x(h^3\partial_{xxx}h),
\label{eq:thif1}
\end{equation}
on the domain $\Omega:=\R\times\R^+$ subject to the initial condition $h(x,0)=h^0(x)$.

The thin-film equation can be used to describe not only extended films deposited on substrates, but also, droplets provided the contact angle is small.  In particular, Equation~\eqref{eq:thif1} has admits a piecewise parabolic equi\-li\-brium solution $h(x)=\max(0,A-Bx^2)$, where $A$ and $B$ are positive constants, corresponding to a stationary droplet which extends from $x=-\sqrt{A/B}$ to $x=\sqrt{A/B}$.  In this context, $x=\pm \sqrt{A/B}$ are triple points, which are points where the fluid, the substrate, and the gas in the surrounding atmosphere meet (such points are also referred to as \textit{contact lines}, which makes geometric sense in the context of a three-dimensional droplet).

The thin-film equation as stated in Equation~\eqref{eq:thif1} fails to properly capture the passage to equilibrium, in particular, the motion of the contact line~\cite{huh1971}. The moving contact line violates the no-slip assumption of viscous fluid flow and manifests as a stress singularity at the contact line. A typical occurrence of this scenario is the droplet spreading. To resolve the paradox, the Equation~\eqref{eq:thif1} needs to be modified in such a way as to resolve the singularity at the contact line while keeping the macro scale dynamics of the thin-film equation. This step is often referred to as regularization. Multiple regularization methods are proposed in the literature~\cites{huh1971, hocking1981, sibley2015, deGennes1985, eggers2004, seppecher1996, ding2007, gtfe2022}.  While these approaches differ, they all rely on the parametrizing of physics at small scales (typically, on the nanometre scale), and give the same qualitative results on the droplet scale (typically, on scales from microns to millimetres).

The present work is concerned with the Geometric Thin-Film Equation, a novel regularization of Equation~\eqref{eq:thif1} introduced previously by the present authors in Reference~\cite{gtfe2022}.  This regularization involves a smoothened free-surface height  $\bar{h}(x,t)$, which is connected to the basic free-surface height $h(x,t)$ via convolution:
\begin{subequations}
\begin{equation}
\bar{h}=K*h.
\label{eqn_convolution}
\end{equation}
By requiring that the dynamics of $h(x,t)$ reduce the surface energy 
\begin{equation}
E=\tfrac{1}{2}\lint{-\infty}{\infty}{\partial_x h \partial_x \bar{h}}{x}
\end{equation}
over time, it has been shown that $h(x,t)$ must satisfy the following equation:
\begin{equation}
    \partial_t h + \partial_x[m(h,\bar{h})\partial_{xxx}\bar{h}] = 0.
    \label{eq:hdef}
\end{equation}
Furthermore, in Reference~\cite{gtfe2022} the authors justify the choice of $K$ as the Green's function for the bi-Helmholtz problem $(1-\alpha^2\partial_{xx})^2K(x)=\delta(x)$, hence $K$ is given by:
\begin{equation}
    \label{eq:kernel}
    K(x)=\frac{1}{4\alpha^2}(\alpha+|x|)\me^{-|x|/\alpha}.
\end{equation}%
\label{eq:hdefall}%
\end{subequations}%
Equation~\eqref{eq:hdefall} is solved on the domain $\Omega$, subject to initial condition $h(x,0)=h^0(x)\geq0$. Also, $m(h,\bar{h})=h\bar{h}^2$ is the mobility, and $\alpha$ is a small positive lengthscale (the smoothing lengthscale).

In \cite{gtfe2022}, the present authors have explored Equation~\eqref{eq:hdefall}  numerically.  There, it was demonstrated that the regularized model reproduces the known sprea\-ding behaviour of droplets with small equilibrium contact angles, specifically, the Cox-Voinov law and the Tanner's Law of spreading.  As such, the aim of the present work is to further characterize weak solutions of Equation~\eqref{eq:hdefall}, this time using analytical techniques.

The Geometric Thin-Film Equation can be viewed as a special case of a mechanical model for energy-dissipation on general configuration spaces—the derivation of the general model involves methods from Geometric Mechanics such as Lie Derivatives and Momentum Maps~\cite{holm2008geometric} -- hence the name Geo\-me\-tric Thin-Film Equation. The main advantage of this new method so far has been the non-stiff nature of the differential equations in the model, which leads to robust numerical simulation results. A second advantage is that the Geometric Thin-Film equation admits so-called particle solutions. These give rise to an efficient and accurate numerical method (the particle method) for solving the model equations. Moreover, they can be used to establish weak solutions of \eqref{eq:hdefall} in the general case.

\subsection*{Characteristics and Particle solutions}

Equation~\eqref{eq:hdefall} can be further written as
\begin{subequations}
\begin{equation}
    \partial_t h + \partial_x(vh) = 0.
    \label{eq:hdef_v}
\end{equation}
where
\begin{equation}
v(h(x,t))=\bar{h}^2\partial_{xxx}\bar{h}.
\label{eq:vdef}%
\end{equation}%
\label{eq:hdef_v_all}%
\end{subequations}%
Here, $v$ has the interpretation of a transport velocity, and  Equation~\eqref{eq:hdef_v} is a scalar density equation.  The dependence of $v$ on $x$ and $t$ is implicit through $h(x,t)$; this is made clear in Equation~\eqref{eq:vdef}. In this way, the mass-conservation property of Geometric Thin-Film Equation is also made clear:
\[
\frac{\mathrm{d}}{\mathrm{d} t}\int_{-\infty}^\infty h(x,t)\mathrm{d} x=0,
\]
subject to appropriate boundary conditions on $h(x,t)$ as $|x|\rightarrow \infty$.  

Based on this treatment, we can furthermore identify the characteristic equation:
\begin{subequations}
\begin{equation}
\frac{\mathrm{d} x(t)}{\mathrm{d} t}=v(h(x(t),t)),\qquad t>0
\end{equation}
subject to the initial condition
\begin{equation}
x(t=0)=x^0.
\end{equation}%
\label{eq:chardef}%
\end{subequations}%
In Sections~\ref{sec:coll}--\ref{sect_uniqueness}, we prove that the continuous family of ordinary differential equations~\eqref{eq:chardef} (paramet\-rized by the continuous parameter $x^0$) possesses unique global-in-time solutions and hence, there exists a flow:
\begin{eqnarray}
c:\R \times \R^+&\rightarrow & \mathbb{R},\nonumber \\
              (x^0,t)&\mapsto& c(x^0,t),
\end{eqnarray}
such that
\begin{eqnarray*}
    c_t(x^0,t)&=&v(h(c(x^0,t)),t),\qquad t>0\\
    c(x^0,0)&=&x^0.
\end{eqnarray*}
The global-in-time result is established by proving that characteristic curves do not cross, specifically, if $x^0$ and $y^0$ are two initial conditions such that $x^0<y^0$, then
\begin{equation}
c(x^0,t)<c(y^0,t).
\label{eq:cxy0}
\end{equation}
By this \textit{no-crossing theorem}, any point $x$ at any time $t>0$ can be traced back unambiguously to a single starting-point $x^0$ and hence, the flow $c(x^0,t)$ is invertible, with inverse $\eta$, such that $x^0=\eta(x,t)$.  As such, the solution $h(\cdot,t)$ of Equation~\eqref{eq:hdef_v_all} can be identified as a push-forward of measures, with:
\begin{equation}
h(x,t)=h^0(\eta(x,t))\frac{\partial\eta}{\partial x}.
\label{eq:pfm}
\end{equation}
In our treatment, in full generality $h(\cdot,t)$ and $h^0$ take the form of positive Radon measures in $\mathcal{M}(\R)$ (that are not necessarily absolutely continuous with respect to Lebesgue measure). 
We make these ideas more precise in the main part of the paper.

The existence of solutions to the characteristic equation~\eqref{eq:chardef} can be established by constructing approximate `particle' solutions and taking a limit. One of the key properties of Equation~\eqref{eq:hdefall} is that it admits weak solutions expressible as finite weighted sums of Dirac delta distributions/measures:
\begin{equation}
    h(x,t) = \sum_{i=1}^N w_i \delta(x-x_i(t)) = \sum_{i=1}^N w_i \delta_{x_i(t)}.
    \label{eq:hNdef}
\end{equation}
Each weight $w_i$ and position $x_i(t)$ is identified with a pseudo-particle -- hence, solutions of this form are referred to as particle solutions.
%
Substituting Equation~\eqref{eq:hNdef} into a weak form of Equation~\eqref{eq:hdefall} (see \eqref{eq:weak} below) yields a system of $N$ first order ODEs given by
\begin{subequations}
\label{eq:ivp}
\begin{equation}
    \label{eq:ivp1}
    \begin{cases}
        \dot{x}_i(t) \!\!\!\!&= v_i(x_1(t),\cdots,x_N(t)),\\
        x_i(0) \!\!\!\!&= x_i^0,
    \end{cases} 
    \qquad i = 1,\dots,N.
\end{equation}
where
\begin{equation}
\label{eq:ivp2}
v_i(x_1(t),\cdots,x_N(t))=\bar{h}(x_i(t),t)^2\partial_{xxx}\bar{h}(x_i(t),t), 
\end{equation}
\end{subequations}
and  where we have used a slight abuse of notation, with
\begin{align}
    \bar{h}(x_i(t),t) &= \sum_{j=1}^N w_j K(x_i(t)-x_j(t)), \\
    \partial_{xxx}\bar{h}(x_i(t),t) &= \sum_{\substack{j=1\\j\neq i}}^N w_j K'''(x_i(t)-x_j(t)). \label{eqn_3rd_deriv}
\end{align}
We note that $K$ is only twice differentiable in the classical sense and that $K'''$ is not defined at the origin.  However, one of the key results in the main part of the paper is that a solution that avoids this discontinuous point at $t=0$ stays away from this discontinuous point for all later times. That we can disregard the fact that $K'''(0)$ is undefined and remove the index $j=i$ from the right hand side of \eqref{eqn_3rd_deriv} is justified rigorously in Theorem \ref{theorem_ODE}.

Once the existence of particle solutions has been established, the idea is to form a sequence of particle solutions (parametrized by the discrete parameters $x_1^0,\cdots,x_N^0$) that converge in a suitable sense to the flow $c(x^0,t)$ (parametrized by the continuous parameter $x^0$).  This then enables a construction of a weak solution of Equation~\eqref{eq:hdefall} by the push-forward measure~\eqref{eq:pfm}.  Furthermore, we will demonstrate that the solution is unique, with respect to the class of possible solutions expressible as such push-forward measures. {In addition, using some of the theory developed in \cite{onaraighpangsmith2022}, we determine the regularity of the smoothened free-surface height $\bar{h}$.

\subsection*{Literature Review}

The general approach we take to demonstrate the existence of solutions via push-forward measures has a long-established history. It can be traced back to \cites{braunkepp1977, dobrushin1979}, wherein the existence and uniqueness of solutions of Vlasov equations are demonstrated by means of $w^*$-convergent sequences of measures, distances of Kantorovich-Rubinstein-Wasserstein type (in this case the bounded Lipschitz distance) and the notion of couplings of pairs of measures from optimal transport theory; see also e.g.~\cite{spohn1991}*{Chapter 5} and \cite{golse:16}*{Chapter 1}. Many refinements of these techniques and applications to models can be found in the literature, e.g.~ \cites{ccr2011,cchs2019,cdft2007,carrillo2019,carrillotoscani2004,carrillotoscani2007,litoscani2004,toscani2000}.

The particle solutions which we use here as an intermediate step in our existence proof are a key feature of other continuum models. In the Camassa--Holm (CH) equation (a model for waves in shallow water), the particles are referred to as `peakons' and are analogous to solitons found in other water-wave models.  Peakons are weak solutions of the CH equation; the peakon positions and momenta satisfy a set of ODEs with a Hamiltonian structure.  As such, the authors of Reference~\cite{chertock2012a} have shown the convergence of particle solutions to a weak solution of the CH equation, the main theoretical tool used here is Helly's selection theorem.  The present authors have generalized this approach using a metric Arzel\`a-Ascoli compactness theorem~\cite{onaraighpangsmith2022}. 
Other continuum models also admit particle solutions.  In \cite{laurent2007}, the author has analysed the interval of existence of an aggregation equation and provides an ``acceptability" condition for the kernel function for which the solution exhibits finite time blow-up behaviour.  Also, in \cite{carrillo2019}, the authors have introduced  smoothened diffusion- and porous-medium equations, which admit particle solutions.  The particle solutions then give rise to a novel method for performing numerical simulations of such continuum models. Our approach to the present problem is similar to these existing works.

The particle solutions of the Geometric Thin-Film equation~\eqref{eq:hdefall} have already been studied numerically by the present authors. In particular, in \cite{gtfe2022}, we introduced a fast summation algorithm to reduce the numerical cost of evaluating the system of ODEs in Equation~\eqref{eq:ivp} from $O(N^2)$ to $O(N)$. The algorithm relies on an assumption that the relative ordering of the particles is preserved throughout the interval of the existence of the solution, in other words, crossings (or collisions) are not permitted between the particles. One of the aims of the present work is to prove that this assumption is valid: the particles do not cross in finite time.

We emphasize that other models which admit particle solutions do possess such a `no-crossing theorem' (e.g. the CH equation~\cites{camassa2006, chertock2012b}).  In the case of the CH equation, the no-crossing theorem is proved using arguments about the conservation of total momentum, these arguments rely on the underlying Hamiltonian structure of the CH equation.  However, in the present case, there is no such Hamiltonian structure; instead, the no-crossing is shown to be a consequence of a quantitative estimate established by analysing $K$ and $K'''$, and applying a form of Gr\"onwall's inequality.

\subsection*{Plan of the paper}

The purpose of Section \ref{sec:weak} is to state a weak form of \eqref{eq:hdefall}, identify a class of potential weak solutions expressible as push-forwards of the initial condition, and finally specify carefully the corresponding ODE which yields a weak solution of \eqref{eq:hdefall}.  This is Equation~\eqref{eq:chardef} again, albeit now stated more precisely.   In Sections \ref{sec:coll} and \ref{sect_existence}, we show that there exist solutions of the ODE which are globally-defined in time.  Moreover, we show that the smoothened free-surface height $\bar{h}$ has the desired regularity properties. In Section \ref{sect_uniqueness}, we show that there is only one such solution with the given initial data, with respect to this class. Finally, in Section \ref{sect_final_remarks}, we remark on the continuity properties of the solution and introduce a wider class of weak solutions with respect to which uniqueness fails. In so doing, we provide some justification for considering the class of push-forward solutions.

\section{Weak formulation of the Geometric Thin-Film Equation}\label{sec:weak}

In this section we introduce the weak formulation of the Geometric Thin-Film Equation, a class of potential weak solutions expressible as push-forwards of the initial condition, and finally a sufficient condition for an element of this class to be a weak solution. 

Let $\mathcal{M}(\R)$ denote the space of Radon measures (equivalently finite Borel measures) on $\R$, equipped with the total variation norm $\pndot{1}$, and let $\mathcal{M}^+(\R)$ denote the subset of positive measures. As stated above, in full generality we will treat the function $h$ as positive measure-valued, i.e. we can regard $h$ as a map from $\R^+$ to $\mathcal{M}^+(\R)$ with $h(0)=h^0$. It will help to make the identification $\bar{h}(t) = \bar{h}(\cdot,t) = K * h(t)$. We can write $h=(1-\alpha^2\partial_{xx})^2\bar{h}$, where the derivative on the right hand side is taken in a distributional sense, and is representable as a measure. With these relationships in mind, we make the following definition.

\begin{definition}\label{defn_weak_solution}
We say that the pair $(h,\bar{h})$ is a weak solution of Equation~\eqref{eq:hdefall} if it satisfies
\begin{multline}
\lint{-\infty}{\infty}{\phi(x,0)}{h^0(x)} + \lint{0}{\infty}{ \lint{-\infty}{\infty}{\bar{h} (1-\alpha^2\partial_{xx})^2\phi_t}{x}}{t} + \lint{0}{\infty}{\lint{-\infty}{\infty}{\bar{h}^3 \partial_{xxx}\bar{h} \phi_x}{x}}{t} \\
-2\alpha^2 \lint{0}{\infty}{\lint{-\infty}{\infty}{\bar{h}^2 \partial_{xx}\bar{h} \partial_{xxx}\bar{h} \phi_x}{x}}{t} - \alpha^4 \lint{0}{\infty}{\lint{-\infty}{\infty}{\bar{h} \partial_{x}\bar{h} (\partial_{xxx}\bar{h})^2 \phi_x}{x}}{t} \\
-\frac{1}{2}\alpha^4 \lint{0}{\infty}{\lint{-\infty}{\infty}{(\bar{h} \partial_{xxx}\bar{h})^2 \phi_{xx}}{x}}{t} = 0,
\label{eq:weak}
\end{multline}
for all indefinitely differentiable and compactly-supported test functions $\phi\in C_c^\infty (\Omega)$.
\end{definition}

Next, we introduce our class of potential solutions of \eqref{eq:weak}. Let $\mu \in \mathcal{M}^+(\R)\setminus\{0\}$ satisfy $\pn{\mu}{1} \leq 1$. To make the notation a little easier, we will replace our initial condition $h^0$ by $\mu$. Let $M \subseteq \R$ be a set that is finite or equal to $\R$, and having the property that $\supp \mu \subseteq M$, where $\supp \mu$ denotes the support of $\mu$.

We will build a system of curves and corresponding push-forward measures. Consider the vector space $X_M$ of locally bounded functions $\map{c}{M \times \R^+}{\R}$, such that $c(\cdot,t)$ is Borel measurable for all $t \in \R^+$ and $c(x,\cdot)$ is $C^1$ for all $x \in M$. We define $h(t)$, $t \in \R^+$, as the push-forward measure $h(t) = c(\cdot,t)_*\mu$. The initial condition $h(0)=\mu$ is established by requiring that $c(x,0)=x$, $x \in M$. Following \eqref{eq:cxy0}, in addition we require the following `no-crossing' condition:
\begin{equation}\label{eqn_no-crossing}
 c(x,t) < c(y,t) \text{ whenever }x,y \in M,\, x < y,
\end{equation}
for all $t \in \R^+$. This is equivalent to saying that each function $c(\cdot,t)$ is strictly increasing on $M$ (this condition ensures that $c(\cdot,t)$ has at most countably many discontinuities, so certainly it is Borel measurable). For this reason we introduce the set
\[
I_M = \set{c \in X_M}{c \text{ satisfies \eqref{eqn_no-crossing} for all $t \in \R^+$}}.
\]

The function $h$ is supposed to provide a solution of \eqref{eq:weak}. In order to do this, we identify in \eqref{eqn_ODE_system} below a corresponding system of ODEs for the curves $c(x,\cdot)$ which, when satisfied, induces a solution $h$ of \eqref{eq:weak}. This is our precise verison of \eqref{eq:chardef}. Before the theorem is presented, a little groundwork is needed. First, we make a remark about $\bar{h}(t) = \bar{h}(\cdot,t)$. As shown in \eqref{eqn_convolution}, this function is defined by convolving $K$ with $h(t)$, with the latter now dependent on $c(\cdot,t)$. In particular, we observe the identity
\begin{align}
\bar{h}(c(x,t),t) \;=& \lint{-\infty}{\infty}{K(c(x,t)-z)}{h(t)(z)} \nonumber\\
=& \lint{-\infty}{\infty}{K(c(x,t)-c(z,t))}{\mu(z)}. \label{eqn_H}
\end{align}
To make this dependence on $c$ explicit, we will write $H(c,x,t) = \bar{h}(c(x,t),t)$.

Second, we present the following simple lemma.

\begin{lemma}\label{lem_K'''}
Given $x \neq 0$ we have
\begin{equation}\label{eqn_K'''}
K(x)K'''(x) - 2\alpha^2 K''(x)K'''(x) + \ts \frac{1}{2}\alpha^4((K''')^2)'(x) = 0.
\end{equation}
\end{lemma}

\begin{proof}
 It is straightforward to verify that for $x \neq 0$
 \[
  K''(x) = \frac{1}{4\alpha^2}\left( \frac{|x|}{\alpha^2} - \frac{1}{\alpha}\right) \me^{-\frac{|x|}{\alpha}}, \qquad
  K'''(x) = \frac{1}{4\alpha^2}\left( \frac{2\sgn(x)}{\alpha^2} - \frac{x}{\alpha^3}\right) \me^{-\frac{|x|}{\alpha}}
  \]
and
\[
((K''')^2)'(x) = \frac{1}{16\alpha^4}\left( -\frac{12\sgn(x)}{\alpha^5} + \frac{10x}{\alpha^6} - \frac{2x|x|}{\alpha^7}\right) \me^{-\frac{2|x|}{\alpha}}.
\]
From these identities the conclusion readily follows.
\end{proof}
Third, we set down some notation. Let $\Delta=\set{(x,x)}{x \in \R}$ denote the diagonal of $\R^2$ and let $\hat{h}(t)$ denote the product of the measure $h(t)$ with itself, thus, given a bounded Borel measurable function $\map{f}{\R^2}{\R}$ we have
\[
 \lint{\R^2}{}{f(x,y)}{\hat{h}(t)(x,y)} = \lint{-\infty}{\infty}{\lint{-\infty}{\infty}{f(x,y)}{h(t)(x)}}{h(t)(y)}.
\]

Given $T > 0$, consider the vector space $X_{M,T}$ of locally bounded functions $\map{c}{M \times [0,T]}{\R}$, such that $c(\cdot,t)$ is Borel measurable for all $t \in [0,T]$ and $c(x,\cdot)$ is $C^1$ for all $x \in M$. In addition we define the subset
\[
I_{M,T} = \set{c \in X_{M,T}}{c \text{ satisfies \eqref{eqn_no-crossing} for all $t \in [0,T]$}}.
\]
For convenience we will allow $T=\infty$ and write $X_{M,\infty}=X_M$ and $I_{M,\infty} = I_M$.

Finally we will specify our precise version of \eqref{eq:chardef}. Given $c \in I_{M,T}$, consider the (finite or continuous) system of ODEs given by 

\begin{equation}\label{eqn_ODE_system}
 \begin{cases}
  c_t(x,t) \!\!\!\!&= \ds H(c,x,t)^2\lint{z \neq x}{}{K'''(c(x,t)-c(z,t))}{\mu(z)}, \\
  c(x,0) \!\!\!\!&= x,
 \end{cases} \qquad (x,t) \in M \times (0,T).
 \end{equation}

Now we are able to state our sufficient condition.

\begin{theorem}\label{theorem_ODE}
Let $c \in I_M$ satisfy \eqref{eqn_ODE_system} (where $T=\infty$). Then the corresponding map $h$ of push-forward measures is a solution of \eqref{eq:weak}.
\end{theorem}

\begin{proof}
Let $\phi \in C^\infty_c(\Omega)$ be a test function. Define
\begin{equation}\label{eqn_T_1_1}
 T_1 = \lint{-\infty}{\infty}{\phi(x,0)}{\mu(x)}
\end{equation}
and
\begin{equation}\label{eqn_T_2_1}
 T_2 = \lint{0}{\infty}{ \lint{-\infty}{\infty}{\bar{h} (1-\alpha^2\partial_{xx})^2\phi_t}{x}}{t} = \lint{0}{\infty}{\lint{-\infty}{\infty}{\phi_t(x,t)}{h(t)(x)}}{t},         
\end{equation}
and let $T_3$ denote the remaining terms in \eqref{eq:weak}. To be a weak solution, we require that $T_1+T_2+T_3=0$. This will be done in two steps. In the first step we find an alternative expression for $T_3$. We write
\[
 T_3 = \lint{0}{\infty}{I_1 + I_2}{t},
\]
where
\[
I_1 \;:=\; \lint{-\infty}{\infty}{\left\{\bar{h}^2(\bar{h}\partial_{xxx}\bar{h} - 2\alpha^2 \partial_{xx}\bar{h}\partial_{xxx}\bar{h}) - \alpha^4\bar{h}\partial_x\bar{h}(\partial_{xxx}\bar{h})^2 \right\}\phi_x}{x}
\]
and
\[
I_2 \;:=\; {\ts -\frac{1}{2}\alpha^4} \lint{-\infty}{\infty}{\bar{h}^2(\partial_{xxx}\bar{h})^2 \phi_{xx}}{x}.
\]
The alternative expression for $T_3$ depends in turn on finding a different way to express the sum $I_1+I_2$. For now, we suppress the dependency of $h$ and $\phi$ on $t$. In the following computations, $\lint{\Delta}{}{\cdot}{\hat{h}(y,y)}$ denotes integration of the product measure along the diagonal $\Delta$. We write
\begin{align}
I_1 \;=& \int_{-\infty}^{\infty} \bar{h}(x)^2 \left\{ \lint{-\infty}{\infty}{K(x-y)}{h(y)}\lint{-\infty}{\infty}{K'''(x-z)}{h(z)} \right. \nonumber\\
& \left. - 2\alpha^2 \lint{-\infty}{\infty}{K''(x-y)}{h(y)}\lint{-\infty}{\infty}{K'''(x-z)}{h(z)} \right\}\phi_x\, \md x \nonumber\\
& -\alpha^4 \lint{-\infty}{\infty}{\bar{h}(x)\partial_x \bar{h}(x)\lint{-\infty}{\infty}{K'''(x-y)}{h(y)}\lint{-\infty}{\infty}{K'''(x-z)}{h(z)}\phi_x}{x} \nonumber\\
=& \lint{\Delta}{}{\lint{-\infty}{\infty}{\bar{h}(x)^2\left\{K(x-y)K'''(x-y) - 2\alpha^2 K''(x-y)K'''(x-y) \right\}\phi_x}{x}}{\hat{h}(y,y)} \nonumber\\
& -\underbrace{\alpha^4 \lint{\Delta}{}{\lint{-\infty}{\infty}{\bar{h}(x)\partial_x \bar{h}(x)K'''(x-y)^2\phi_x}{x}}{\hat{h}(y,y)}}_{=:\,J_1} \nonumber\\
& +\lint{\R^2\setminus\Delta}{}{\lint{-\infty}{\infty}{\bar{h}(x)^2\left\{K(x-y)K'''(x-z) - 2\alpha^2 K''(x-y)K'''(x-z) \right\}\phi_x}{x}}{\hat{h}(y,z)} \nonumber\\
& -\underbrace{\alpha^4 \lint{\R^2\setminus\Delta}{}{\lint{-\infty}{\infty}{\bar{h}(x)\partial_x \bar{h}(x)K'''(x-y)K'''(x-z)\phi_x}{x}}{\hat{h}(y,z)}}_{=:\,J_2}, \label{eqn_new_I_1}
\end{align}
and
\begin{align*}
I_2 \;=& {\ts -\frac{1}{2}\alpha^4}\lint{-\infty}{\infty}{\bar{h}(x)^2 \left\{ \lint{-\infty}{\infty}{K'''(x-y)}{h(y)}\lint{-\infty}{\infty}{K'''(x-z)}{h(z)} \right\} \phi_{xx}}{x}\\
=& {\ts -\frac{1}{2}\alpha^4} \lint{\Delta}{}{\left\{\lint{-\infty}{\infty}{\bar{h}(x)^2K'''(x-y)^2\phi_{xx}}{x} \right\}}{\hat{h}(y,y)}\\
& {\ts -\frac{1}{2}\alpha^4} \lint{\R^2\setminus\Delta}{}{\left\{\lint{-\infty}{\infty}{\bar{h}(x)^2K'''(x-y)K'''(x-z)\phi_{xx}}{x} \right\}}{\hat{h}(y,z)}.
\end{align*}
Applying integration by parts to the inner integrals and swapping indices yields
\begin{align}
I_2 \;=& \underbrace{\alpha^4 \lint{\Delta}{}{\lint{-\infty}{\infty}{\bar{h}(x)\partial_x \bar{h}(x)K'''(x-y)^2\phi_x}{x}}{\hat{h}(y,y)}}_{=\,J_1} \nonumber\\
& +{\ts \frac{1}{2}\alpha^4} \lint{\Delta}{}{\lint{-\infty}{\infty}{\bar{h}(x)^2\partial_x (K'''(x-y)^2)\phi_x}{x}}{\hat{h}(y,y)} \nonumber\\
& +\underbrace{\alpha^4 \lint{\R^2\setminus\Delta}{}{\lint{-\infty}{\infty}{\bar{h}(x)\partial_x \bar{h}(x)K'''(x-y)K'''(x-z)\phi_x}{x}}{\hat{h}(y,z)}}_{=\,J_2} \nonumber\\
& +\alpha^4 \lint{\R^2\setminus\Delta}{}{\lint{-\infty}{\infty}{\bar{h}(x)^2K''''(x-y)K'''(x-z)\phi_x}{x}}{\hat{h}(y,z)}. \label{eqn_new_I_2}
\end{align}
When we sum $I_1$ and $I_2$ the $J_1$ and $J_2$ terms in \eqref{eqn_new_I_1} and \eqref{eqn_new_I_2} cancel, giving
\begin{align}
 I_1+I_2 \;=& \int_{-\infty}^\infty \int_{-\infty}^\infty \bar{h}(x)^2 \left\{ K(x-y)K'''(x-y) - 2\alpha^2 K''(x-y)K'''(x-y) \right. \nonumber\\
 & +\left. {\ts \frac{1}{2}\alpha^4} \partial_x (K'''(x-y)^2) \right\}\phi_x \,\md x\, \md \hat{h}(y,y)\nonumber\\
 & +\int_{\R^2\setminus\Delta} \int_{-\infty}^\infty \bar{h}(x)^2 \left\{K(x-y) - 2\alpha^2K''(x-y)\right.\nonumber\\
 & +\left. \alpha^4 K''''(x-y) \right\} K'''(x-z)\phi_x \,\md x\,\md \hat{h}(y,z). \label{eqn_I_1+I_2}
\end{align}
By Lemma \ref{lem_K'''} the first integral on the right hand side of \eqref{eqn_I_1+I_2} vanishes, and as
\[
 K(x) - 2\alpha^2 K''(x) + \alpha^4 K''''(x) = \delta(x),
\]
the second integral can be simplified, giving
\begin{align*}
I_1 + I_2 \;=& \lint{\R^2\setminus\Delta}{}{\lint{-\infty}{\infty}{\bar{h}(x)^2\delta(x-y)K'''(x-z)\phi_x(x)}{x}}{\hat{h}(y,z)}\\
 =& \lint{\R^2\setminus\Delta}{}{\bar{h}(y)^2 K'''(y-z)\phi_x(y)}{\hat{h}(y,z)}.
\end{align*}
We now reinstate the dependence of $h$ and $\phi$ on $t$ and write
\begin{equation}\label{eqn_T_3_1}
 T_3 = \lint{0}{\infty}{I_1+I_2}{t} = \lint{0}{\infty}{\lint{\R^2\setminus\Delta}{}{\bar{h}(x,t)^2K'''(x-z)\phi_x(x,t)}{\hat{h}(t)(x,z)}}{t}.
\end{equation}
This completes the first step of the proof.

In the second step, we show that $T_1+T_2+T_3=0$. Using the push-forward definition of $h(t)$, we have
\begin{align*}
 & \lint{\R^2\setminus\Delta}{}{\bar{h}(x,t)^2K'''(x-z)\phi_x(x,t)}{\hat{h}(t)(x,z)} \\
 =& \lint{-\infty}{\infty}{\bar{h}(x,t)^2\phi_x(x,t)\lint{z \neq x}{}{K'''(x-z)}{h(t)(z)}}{h(t)(x)} \\
  =& \lint{-\infty}{\infty}{\bar{h}(x,t)^2\phi_x(x,t)\lint{c(z,t) \neq x}{}{K'''(x-c(z,t))}{\mu(z)}}{h(t)(x)} \\
  =& \lint{-\infty}{\infty}{H(c,x,t)^2\phi_x(c(x,t),t)\lint{c(z,t) \neq c(x,t)}{}{K'''(c(x,t)-c(z,t))}{\mu(z)}}{\mu(x)}\\
  =& \lint{-\infty}{\infty}{H(c,x,t)^2\phi_x(c(x,t),t)\lint{z \neq x}{}{K'''(c(x,t)-c(z,t))}{\mu(z)}}{\mu(x)},
\end{align*}
noting that, by \eqref{eqn_no-crossing}, $z \neq x$ if and only if $c(z,t) \neq c(x,t)$. Thus by this and \eqref{eqn_T_3_1} we get
\begin{equation}\label{eqn_T_3_2}
 T_3 = \lint{0}{\infty}{\lint{-\infty}{\infty}{H(c,x,t)^2\phi_x(c(x,t),t)\lint{z \neq x}{}{K'''(c(x,t)-c(z,t))}{\mu(z)}}{\mu(x)}}{t}.
\end{equation}
Now recall the definitions of $T_1$ and $T_2$ in \eqref{eqn_T_1_1} and \eqref{eqn_T_2_1}. Define $\psi(x,t) = \phi(c(x,t),t)$. Then 
\[
 \psi_t(x,t) = \phi_t(c(x,t),t) + \phi_x(c(x,t),t)c_t(x,t),
\]
hence we can write
\begin{align*}
 \lint{-\infty}{\infty}{\phi_t(x,t)}{h(t)(x)} \;=& \lint{-\infty}{\infty}{\phi_t(c(x,t),t)}{\mu(x)} \\
 =& \lint{-\infty}{\infty}{\psi_t(x,t)}{\mu(x)} - \lint{-\infty}{\infty}{\phi_x(c(x,t),t)c_t(x,t)}{\mu(x)}.
\end{align*}
Substituting this into \eqref{eqn_T_2_1} and recalling \eqref{eqn_T_1_1} gives
\begin{align}
 T_2 \;=& \lint{-\infty}{\infty}{\lint{0}{\infty}{\psi_t(x,t)}{t}}{\mu(x)} - \lint{0}{\infty}{\lint{-\infty}{\infty}{\phi_x(c(x,t),t)c_t(x,t)}{\mu(x)}}{t} \nonumber\\
 =& -\lint{-\infty}{\infty}{\psi(x,0)}{\mu(x)} - \lint{0}{\infty}{\lint{-\infty}{\infty}{\phi_x(c(x,t),t)c_t(x,t)}{\mu(x)}}{t} \nonumber\\
 =& {-T_1} - \lint{0}{\infty}{\lint{-\infty}{\infty}{\phi_x(c(x,t),t)c_t(x,t)}{\mu(x)}}{t}. \label{eqn_T_2_2}
\end{align}
Finally, recalling \eqref{eqn_T_3_2}, substituting \eqref{eqn_ODE_system} into \eqref{eqn_T_2_2} and rearranging gives
\begin{align*}
 & T_1+T_2\\
 =& {-\lint{0}{\infty}{\lint{-\infty}{\infty}{H(c,x,t)^2\phi_x(c(x,t),t)\lint{z \neq x}{}{K'''(c(x,t)-c(z,t))}{\mu(z)}}{\mu(x)}}{t}}\\
 =& {-T_3},
\end{align*}
as required.
\end{proof}

\begin{remark}\label{rem_theorem_ODE_converse}
Theorem \ref{theorem_ODE} has a partial converse. If $c \in I_M$ is a solution of \eqref{eq:weak} then it can be shown that \eqref{eqn_ODE_system} is satisfied whenever $x \in \supp\mu$. The curves $c(x,\cdot)$, $x \in \supp \mu$, are the only ones that matter in the definition of $h$, but it will be technically convenient for us to assume that \eqref{eqn_ODE_system} is satisfied for all $x \in M\setminus \supp \mu$ as well. 
\end{remark}

\section{Global existence of particle solutions}\label{sec:coll}

We demonstrate existence of solutions $c \in I_M$ of \eqref{eqn_ODE_system} (and thus solutions of \eqref{eq:weak}) in two steps. The first step involves demonstrating the global existence of `particle solutions'. Before presenting this first step in full detail, we state and prove a key lemma that provides some important estimates of the evolution of a solution $c \in I_M$ of \eqref{eqn_ODE_system}. Its first application will be to show that our particle solutions exist for all $t \in \R^+$.

Lemma \ref{lem_Phi_est_2} uses the following terminology. Noting that $K'''$ is an odd function that is Lipschitz on $(0,\infty)$ and $(-\infty,0)$, let $\Lip(K''')$ denote the common Lipschitz constant of $K'''\restrict{(0,\infty)}$ and $K'''\restrict{(-\infty,0)}$. The following constant will be used throughout the paper. We set
\begin{equation}\label{eqn_A_2}
 A = 2\pn{K}{\infty}(2\Lip(K)\pn{K'''}{\infty} + \pn{K}{\infty}\Lip(K''')).
\end{equation}
Given $x,y \in M$, $x
\leq y$, we set
\begin{equation}\label{eqn_Gamma}
 \Gamma_{x,y} = \pn{K}{\infty}^2\pn{K'''}{\infty}(\mu[x,y) + \mu(x,y]),
\end{equation}
and if $y < x$ set $\Gamma_{x,y}=\Gamma_{y,x}$. It will be helpful to note that $(x,y) \mapsto \Gamma_{x,y}$ defines a pseudometric on $\R$, i.e.~this map is non-negative, symmetric and satisfies the triangle inequality, but $\Gamma_{x,y}=0$ does not imply $x=y$.

In addition to the $\Gamma_{x,y}$, we define
\[
 \Lambda_{x,y} = \pn{K}{\infty}^2 \pn{K'''}{\infty}\max\{\mu(\{x\})^3,\mu(\{y\})^3\}.
\]

\begin{lemma}\label{lem_Phi_est_2}
Fix $T \in (0,\infty]$. Let $x,y \in M$, $x<y$, and let $c \in I_{M,T}$ be a solution of \eqref{eqn_ODE_system}. Then 
 \begin{equation}\label{eqn_c_upper_est}
 c(y,t) - c(x,t) \leq (y-x)\me^{At} + \frac{\Gamma_{x,y}}{A}(\me^{At}-1)
 \end{equation}
 and
 \begin{equation}\label{eqn_c_lower_est}
 (y-x)\me^{-At} + \frac{\Lambda_{x,y}}{A}(1-\me^{-At}) \leq c(y,t) - c(x,t),
 \end{equation}
 for all $t \in [0,T]$ (or $t \in \R^+$ if $T=\infty$).
 \end{lemma}
 
\begin{proof} 
The inequalities are consequences of a generalisation of the standard form of Gr\"onwall's inequality (see e.g.~\cite{schaeffer2016}*{Section 3.5, Exercise 8}), applied to the next two inequalities:
 \begin{equation}\label{eqn_c_t_upper_est}
c_t(y,t) - c_t(x,t) \leq \Gamma_{x,y} + A(c(y,t)-c(x,t))
 \end{equation}
and
\begin{equation}\label{eqn_c_t_lower_est}
\Lambda_{x,y} - A(c(y,t)-c(x,t)) \leq c_t(y,t) - c_t(x,t).
\end{equation}
To prove \eqref{eqn_c_t_upper_est} and \eqref{eqn_c_t_lower_est}, we introduce some notation to make the ensuing calculations easier to read. Given $t \in \R^+$ and $x,y \in M$, set
\[
\Delta(y,x,t) = c(y,t) - c(x,t).
\]
Now fix $x,y \in M$, $x < y$. Using \eqref{eqn_ODE_system} we have
\begin{align}
& c_t(y,t) - c_t(x,t) \nonumber\\
=& (H(c,y,t)^2 - H(c,x,t)^2)\lint{z \neq y}{}{K'''(\Delta(y,z,t))}{\mu(z)} \nonumber\\
 & + H(c,x,t)^2\left(\lint{z \neq y}{}{K'''(\Delta(y,z,t))}{\mu(z)} - \lint{z \neq x}{}{K'''(\Delta(x,z,t))}{\mu(z)}\right). \label{eqn_diff}
\end{align}
As $\pn{\mu}{1} \leq 1$, we have $|H(c,y,t)|,\,|H(c,x,t)| \leq \pn{K}{\infty}$ and
\[
|H(c,y,t)-H(c,x,t)| \leq \Lip(K)\Delta(y,x,t).
\]
Hence the first term on the right hand side of \eqref{eqn_diff} is bounded above by
\begin{equation}\label{eqn_8}
2\Lip(K)\pn{K}{\infty}\pn{K'''}{\infty}\Delta(y,x,t). 
\end{equation}
The difference of integrals in the second term on the right hand side of \eqref{eqn_diff} can be expressed as
\begin{align}
& \lint{z<x}{}{K'''(\Delta(y,z,t)) - K'''(\Delta(x,z,t))}{\mu(z)} \label{eqn_9}\\
+& \lint{y<z}{}{K'''(\Delta(y,z,t)) - K'''(\Delta(x,z,t))}{\mu(z)} \label{eqn_10}\\
+& \lint{x\leq z<y}{}{\big(K'''(\Delta(y,z,t)) - K'''(0)_+\big) + K'''(0)_+}{\mu(z)} \label{eqn_11}\\
-& \lint{x < z \leq y}{}{\big(K'''(\Delta(x,z,t)) - K'''(0)_-\big) + K'''(0)_-}{\mu(z)}. \label{eqn_12}
\end{align}
We estimate the terms labelled \eqref{eqn_9} -- \eqref{eqn_12}. Because $c \in I$ and $K'''\restrict{(0,\infty)}$ is Lipschitz, \eqref{eqn_9} is bounded above by
\begin{equation}
 \Lip(K''')\Delta(y,x,t)\lint{z<x}{}{}{\mu(z)}. \label{eqn_13}
\end{equation}
Likewise, \eqref{eqn_10} is bounded above by
\begin{equation}
 \Lip(K''')\Delta(y,x,t)\lint{y<z}{}{}{\mu(z)}. \label{eqn_14}
\end{equation}
Regarding \eqref{eqn_11}, we have
\begin{align}
\lint{x\leq z<y}{}{K'''(\Delta(y,z,t)) - K'''(0)_+}{\mu(z)} \leq & \Lip(K''')\lint{x\leq z<y}{}{\Delta(y,z,t)}{\mu(z)}\nonumber\\
\leq & \Lip(K''')\Delta(y,x,t)\lint{x\leq z<y}{}{}{\mu(z)}. \label{eqn_15}
\end{align}
Similarly, regarding \eqref{eqn_12}, we have
\begin{align}
\lint{x< z \leq y}{}{K'''(\Delta(x,z,t)) - K'''(0)_-}{\mu(z)} \leq & \Lip(K''')\lint{x< z \leq y}{}{\Delta(z,x,t)}{\mu(z)}\nonumber\\
\leq & \Lip(K''')\Delta(y,x,t)\lint{x< z \leq y}{}{}{\mu(z)}. \label{eqn_16}
\end{align}
Combining estimates \eqref{eqn_13} -- \eqref{eqn_16} and recalling that $K'''$ is odd and $\pn{\mu}{1} \leq 1$, we obtain 
\begin{align}
& \bigg|\lint{z \neq y}{}{K'''(\Delta(y,z,t))}{\mu(z)} - \lint{z \neq x}{}{K'''(\Delta(x,z,t))}{\mu(z)} \nonumber \\
& - K'''(0)_+(\mu[x,y) + \mu(x,y])\bigg| \leq 2\Lip(K''')\Delta(y,x,t). \label{eqn_diff_2}
\end{align}
Therefore, noting that $K'''(0)_+=-K'''(0)_-=\pn{K'''}{\infty}$, \eqref{eqn_diff}, \eqref{eqn_8} and \eqref{eqn_diff_2} combine to give
\begin{align*}
c_t(y,t) - c_t(x,t) \leq& 2\Lip(K)\pn{K}{\infty}\pn{K'''}{\infty}\Delta(y,x,t) + \Gamma_{x,y} + 2\pn{K}{\infty}^2 \Lip(K''')\Delta(y,x,t) \\
\leq & \Gamma_{x,y} + A\Delta(y,x,t),
\end{align*}
yielding \eqref{eqn_c_t_upper_est}.

To obtain \eqref{eqn_c_t_lower_est}, observe that for all $x$ we have
\[
H(c,x,t) = \lint{-\infty}{\infty}{K(c(x,t)-c(z,t))}{\mu(z)} \geq K(0)\mu(\{x\}) = \pn{K}{\infty}\mu(\{x\}).
\]
Using this inequality, \eqref{eqn_diff}, \eqref{eqn_8} and \eqref{eqn_diff_2}, we arrive at
\begin{align}\label{eqn_c_t_lower_est_1}
c_t(y,t) - c_t(x,t) \geq& -2\Lip(K)\pn{K}{\infty}\pn{K'''}{\infty}\Delta(y,x,t) \nonumber\\
&+ H(c,x,t)^2(K'''(0)_+(\mu[x,y) + \mu(x,y]) - 2\Lip(K''')\Delta(y,x,t))\nonumber\\
\geq& \mu(\{x\})^2 \Gamma_{x,y} - A\Delta(y,x,t).\nonumber\\
\geq& \pn{K}{\infty}^2 \pn{K'''}{\infty}\mu(\{x\})^3 - A\Delta(y,x,t).
\end{align}
By rearranging \eqref{eqn_diff} and repeating the above we obtain
\begin{equation}\label{eqn_c_t_lower_est_2}
 c_t(y,t) - c_t(x,t) \geq \pn{K}{\infty}^2 \pn{K'''}{\infty}\mu(\{y\})^3 - A\Delta(y,x,t).
\end{equation}
Combining \eqref{eqn_c_t_lower_est_1} and \eqref{eqn_c_t_lower_est_2} gives us \eqref{eqn_c_t_lower_est}.
\end{proof}

Now we establish the global existence of particle solutions. For the remainder of this section, we will assume that our starting measure $\mu$ is a non-zero finite weighted sum of Dirac measures, and that $M \supseteq \supp \mu$ is finite. As in the Introduction, we can write $M=\{x^0_1,\dots,x^0_N\}$ where $x^0_1 < \dots < x^0_N$ and $\mu=\sum_{i=1}^N w_i\delta_{x^0_i}$, where $w_i \geq 0$, $1 \leq i \leq N$ and $0 < W := \sum_{i=1}^N w_i \leq 1$. Then we define $x_i(t)=c(x^0_i,t)$, $t \in \R^+$, $1 \leq i \leq N$. In this way, the push-forward measure $h(t)$ becomes
\[
 h(t) = \sum_{i=1}^N  w_i \delta_{x_i(t)} = \sum_{i=1}^N w_i \delta(\cdot - x_i(t)),
\]
just as in \eqref{eq:hNdef}, and \eqref{eqn_ODE_system} becomes \eqref{eq:ivp}. 

To obtain a unique solution of \eqref{eq:ivp} that exists for all $t \in \R^+$, first we show that the velocity field above is Lipschitz on a suitable open subset of $\R^N$.

\begin{lemma}\label{lem_velocity_field}
Let $D=\{\bm{x}\in\R^N: x_i<x_j, i<j\}$, and let $f:(0,\infty)\rightarrow\R$ and $g:(-\infty,0)\rightarrow\R$ be Lipschitz continuous. Let $\bm{v}:D\rightarrow\R^N$ be defined by
\begin{equation}
    v_i(\bm{x}) = \sum_{j=1}^{i-1} w_j f(x_i-x_j) + \sum_{j=i+1}^{N} w_j g(x_i-x_j).
\end{equation}
Then $\bm{v}$ is Lipschitz on $D$.
\end{lemma}
\begin{proof}
Let $\bm{x},\bm{y}\in D$, 
\begin{align*}
    \|\bm{v}(\bm{x}) - \bm{v}(\bm{y})\|_1 &= \sum_{i=1}^N |v_i(\bm{x}) - v_i(\bm{y})|, \\
    &\leq \sum_{i=1}^N \sum_{j=1}^{i-1} w_j|f(x_i-x_j)-f(y_i-y_j)| + \sum_{i=1}^N \sum_{j=i+1}^{N} w_j|g(x_i-x_j)-g(y_i-y_j)|.
\end{align*}
Since $g$ and $f$ are Lipschitz on their respective domains, there exists $L>0$ such that
\begin{align*}
    \|\bm{v}(\bm{x}) - \bm{v}(\bm{y})\|_1 &\leq \sum_{i=1}^N \sum_{\substack{j=1\\j\neq i}}^N w_jL|x_i-x_j-y_i+y_j|, \\ 
    &\leq \sum_{i=1}^N \sum_{\substack{j=1\\j\neq i}}^N w_jL(|x_i-y_i|+|x_j-y_j|) 
    \leq 2LN\|\bm{x}-\bm{y}\|_1. \tag*{\qedhere}
\end{align*}
\end{proof}

\begin{example}\label{ex:lip}
The velocity field defined by
\begin{equation}
    v_i(\bm{x}) = \sum_{\substack{j=1\\ j\neq i}}^N w_j K'''(x_i-x_j) = \sum_{j=1}^{i-1} w_j K'''\restrict{(0,\infty)}(x_i-x_j) + \sum_{j=i+1}^{N} w_j K'''\restrict{(-\infty,0)}(x_i-x_j).
\end{equation}
is Lipschitz on $D$ by Lemma \ref{lem_velocity_field}. Since $\bm{x} \mapsto \sum_{j=1}^N w_jK(x_i-x_j)$ is bounded and Lipschitz on $\R^N\supseteq D$, it follows that the velocity field defined in \eqref{eq:ivp2} is bounded and Lipschitz on $D$.
\end{example}

\begin{proposition}\label{prop_global_particle}
There exists a unique solution $\map{\bm{x}}{\R^+}{D}$ of \eqref{eq:ivp}. Therefore (and equivalently), the map $c \in I_M$ given by $c(x^0_i,t)=x_i(t)$, $t \in \R^+$, $1 \leq i \leq N$, is the unique solution of \eqref{eqn_ODE_system}.
\end{proposition}

\begin{proof}
According to \cite{schaeffer2016}*{Proposition 4.1.1}, there exists a unique maximal solution $\map{\bm{x}}{[0,T)}{D}$ of \eqref{eq:ivp}, where $T \in (0,\infty]$. We claim that $T=\infty$, i.e.~this solution exists for all $t \in \R^+$. For a contradiction, we suppose $T<\infty$. By the Mean Value Theorem, we have
\begin{equation}\label{eqn_Lipschitz_time_particle}
|x_i(t)-x_i(s)| \leq \pn{K}{\infty}^2\pn{K'''}{\infty}|t-s|,
\end{equation}
whenever $s,t \in [0,T)$. From this fact it follows that $\bm{x}$ can be continuously extended to $[0,T]$. Moreover, by Lemma \ref{lem_Phi_est_2}, we have $x_j(t) - x_i(t) \geq (x^0_j-x^0_i)\me^{-AT}$ whenever $t \in [0,T]$ and $1 \leq i < j \leq N$. In particular, $\bm{x}(T) \in D$. However, this contradicts \cite{schaeffer2016}*{Theorem 4.1.2} because the image of $[0,T]$ under $\bm{x}$ is a compact subset of $D$.
\end{proof}

\begin{remark}
 Another way to prove Proposition \ref{prop_global_particle} is to show that $D$ (as given in Lemma \ref{lem_velocity_field}) is a trapping region, and then appeal to e.g.~\cite{schaeffer2016}*{Corollary 4.2.4}. Showing that $D$ is a trapping region doesn't require the full force of Lemma \ref{lem_Phi_est_2} -- this will be used in later sections.
\end{remark}

\section{Global existence and regularity of weak solutions}\label{sect_existence}

Fix $\mu \in \mathcal{M}^+(\R)\setminus\{0\}$ satisfying $\pn{\mu}{1} \leq 1$, and set $I=I_\R$. The vast majority of this section is devoted to proving the existence of a solution $c \in I$ of \eqref{eqn_ODE_system}, corresponding to $\mu$. By Theorem \ref{theorem_ODE}, the associated map $\map{h}{\R^+}{\mathcal{M}^+(\R)}$ of push-forward measures will be a solution of \eqref{eq:weak}.

In order to find $c$, we will construct an appropriate sequence of approximating measures $(\mu_n) \subseteq \mathcal{M}^+(\R)\setminus\{0\}$ having finite supports included in certain finite sets $M_n$, to which we can apply Proposition \ref{prop_global_particle}. This will yield a sequence of curves $c_n \in I_{M_n}$, corresponding to the $\mu_n$, that solve \eqref{eqn_ODE_system}. After a limiting process, we obtain our solution $c \in I$ of \eqref{eqn_ODE_system}, corresponding to $\mu$. 

At the end of the section, we establish the regularity of the corresponding map $\bar{h}$. We will show that $\bar{h}(t) = K * h(t) \in H^3(\R)$ for all $t\in \R^+$, and that $\bar{h}\in C_b^{0,\frac{1}{2}}(\R^+;H^3(\R))$, which is the space of $\frac{1}{2}$-H\"older continuous functions $\map{u}{\R^+}{H^3(\R)}$ that are bounded in the sense that $\sup_{t \in \R^+}\n{u(t)} < \infty$.

We set about establishing existence. The choice of sequence $(\mu_n)$ will require a little care. We must single out those points $x \in M$ whose corresponding singletons are atoms of $\mu$, i.e.~$\mu(\{x\})>0$. Let 
\[
 P \;:=\; \set{x \in M}{\mu(\{x\}) > 0}
\]
denote the set of all such points. This set is at most countably infinite in size. 

The existence result, Theorem \ref{theorem_general_existence}, is preceded by a number of technical statements.

\begin{lemma}\label{lem_accn_points}
 If $x \in \R\setminus P$ then
 \[
  \lim_{y \to x_+} \Gamma_{x,y} = 0 \quad\text{and}\quad \lim_{y \to x_-} \Gamma_{x,y} = 0.
 \]
\end{lemma}

\begin{proof}
Let $x \in \R\setminus P$. Then
\[
 \lim_{y \to x_+} \Gamma_{x,y} = \lim_{y \to x_+} \mu([x,y)) + \mu((x,y]) = \mu(\{x\}) + \mu(\varnothing) = 0.
\]
The other limit is established similarly.
\end{proof}

Set $F=\Q \cup P$. The inclusion of the set $P$ of `problematic points' in $F$ will be very important in what follows. It turns out that our solutions $c$ will have the property that $c(\cdot,t)$ is discontinuous at $x$ if and only if $x \in P$.

Let $(M_n)_{n=1}^\infty$ be any increasing sequence of non-empty finite subsets of $F$, such that $\bigcup_n M_n = F$. We associate with each finite set $M_n$ a finite weighted sum of Dirac measures $\mu_n \in \mathcal{M}^+(\R)\setminus\{0\}$ having support included in $M_n$ and satisfying $\pn{\mu_n}{1} \leq \pn{\mu}{1} \leq 1$. Given $n \in \N$ and $x \in M_n$, set
\[
\omega_n(x) = \begin{cases}
\mu[x,\infty) & \text{if $x = \max M_n$}\\
\mu[x,x') & \text{if $x < \max M_n$ and $x' \in M_n$ is minimal such that $x<x'$}. 
\end{cases}
\]
Then define
\[
 \mu_n = \sum_{x \in M_n} \omega_n(x)\delta_x.
\]
To ensure that each $\mu_n$ is non-zero, we impose an additional modest condition on the $M_n$. Let $y \in \supp \mu$ and fix $x \in F$, $x<y$. We will insist that $x \in M_n$ for all $n$. Then $\mu([x,\infty)) \geq \mu((x,\infty)) > 0$ because $y \in \supp \mu$. This ensures that $\mu_n \neq 0$.

Let $C_0(\R)$ denote the Banach space of continuous functions $\map{f}{\R}{\R}$ that vanish at infinity, equipped with the supremum norm. We endow $\mathcal{M}(\R)$ with the $w^*$-topology via the identification $\mathcal{M}(\R)\equiv C_0(\R)^*$. It is evident from the construction of the $\mu_n$ that $w^*$-$\lim_n \mu_n = \mu$. However, we will require some stronger convergence behaviour which is furnished by the next result.

\begin{lemma}\label{lem_strong_convergence}
Let $U \subseteq \R$ be a union of finitely many open intervals and let $\map{f}{\R}{\R}$ be a bounded map whose points of discontinuity (if any) lie in $F \cup (\R\setminus U)$. Then
 \[
  \lim_n \lint{x \in U}{}{f(x)}{\mu_n(x)} = \lint{x \in U}{}{f(x)}{\mu(x)}.
 \]
\end{lemma}

\begin{proof}
First we dismiss the trivial case $U=\varnothing$, so that hereafter $U\neq\varnothing$. Next we tackle the special case that $U$ is the open interval $(a,b)$. By the density of $F$, there exists $N_0 \in \N$ such that $U \cap M_{N_0}$ is non-empty. Given $n \geq N_0$, define $\map{f_n}{\R}{\R}$ by 
\[
 f_n(x) = \begin{cases}0 &\text{if }x < \min U \cap M_n \\ f(x') &\text{if }\min U \cap M_n \leq x \text{ and }x' \in U \cap M_n \cap (-\infty,x] \text{ is maximal}.
 \end{cases}
\]
We show that $\lim_n f_n(x)=f(x)$ for all $x \in U$. Indeed, let $x \in U$ and $\ep>0$. If $x \in F$ then $x \in M_N$ for some $N \geq N_0$, whence $f_n(x)=f(x)$ whenever $n \geq N$. If $x \notin F$ then by hypothesis $f$ must be continuous at $x$, so there exists $\delta>0$ such that $a \leq x-\delta$ and $|f(x)-f(x')| < \ep$ whenever $|x-x'| < \delta$. Again by the density of $F$, we can find $N \geq N_0$ and $x' \in M_N$ such that $0 \leq x - x' < \delta$. Then $n \geq N$ implies $x' \in U \cap M_n \cap (-\infty,x]$, and thus $|f(x)-f_n(x)| < \ep$. By the Dominated Convergence Theorem we conclude that
\begin{equation}\label{eqn_strong-convergence_1}
 \lim_n \lint{x \in U}{}{f_n(x)}{\mu(x)} = \lint{x \in U}{}{f(x)}{\mu(x)}.
\end{equation}
Given $n \geq N_0$, let $x_n = \max U \cap M_n < b$ and define
\[
y_n = \begin{cases} \min [b,\infty) \cap M_n & \text{if }[b,\infty) \cap M_n\text{ is non-empty} \\
           \infty & \text{otherwise.}
          \end{cases}
\]
Now observe that
\begin{align}
 \left| \lint{x \in U}{}{f(x)}{\mu_n(x)} - \lint{x \in U}{}{f_n(x)}{\mu(x)} \right| \;\leq & |f(x_n)|\big(\mu([x_n,y_n)-\mu([x_n,b))\big) \nonumber\\
 =& |f(x_n)|\mu([x_n,y_n)\setminus[x_n,b)).\label{eqn_strong-convergence_2}
\end{align}
Given \eqref{eqn_strong-convergence_1}, to finish the special case it is sufficient to show that the right hand side of \eqref{eqn_strong-convergence_2} tends to $0$, but this is evident as $(y_n)$ is a decreasing sequence that converges to $b$. This completes the special case.

In the general case, we simply note that union of finitely many open intervals can be written as a union of finitely many pairwise disjoint open intervals, write the integrals as finite sums of integrals over these new intervals, and appeal to the above.
\end{proof}

For the next corollary we need some additional terminology. Following \eqref{eqn_Gamma}, given the measures $\mu_n$ we define $\Gamma_{n,x,y}$ in the corresponding way: given $x,y \in M$, $x \leq y$, we set 
\[
 \Gamma_{n,x,y} = \pn{K}{\infty}^2\pn{K'''}{\infty}(\mu_n[x,y) + \mu_n(x,y]),
\]
and if $y < x$ set $\Gamma_{n,x,y}=\Gamma_{n,y,x}$.

\begin{corollary}\label{lem_Gamma_n}
 Given $x,y \in F$, $x<y$, we have $\lim_n \Gamma_{n,x,y} = \Gamma_{x,y}$.
\end{corollary}

\begin{proof}
We show that $\lim_n \mu_n[x,y) = \mu[x,y)$ and $\lim_n \mu_n(x,y] = \mu(x,y]$. These limits follow immediately by applying Lemma \ref{lem_strong_convergence} to the indicator functions $\ind{[x,y)}$ and $\ind{(x,y]}$, which are bounded functions discontinuous only at $x,y \in F$.
\end{proof}

Our final lemma before Theorem \ref{theorem_general_existence} will be needed to establish the convergence of certain integrals appearing in the proof of the latter. It has echos of the classical theorem of Lusin but is tailored to our particular needs.

\begin{lemma}\label{lem_integral_conv}
 Let $\map{f}{\R}{\R}$ be an increasing map whose points of discontinuity (if any) lie in $F$. Then given $\ep>0$, there exists $N \in \N$ and a union $U$ of finitely many open intervals such that
 \begin{enumerate}
  \item $\mu(U) < \ep$;
  \item given $y \in \R\setminus U$, there exist $x,z \in M_N$ such that $x \leq y \leq z$ and $f(z)-f(x) < \ep$.
 \end{enumerate}
 \end{lemma}

 \begin{proof}
  Fix $J \geq 0$ large enough so that $\mu(\R\setminus[-J,J]) < \frac{1}{2}\ep$. Let $N_1,N_2 \in \Z$ such that $\ep N_1 \leq f(-J)$ and $f(J) < \ep N_2$. Given $k \in \Z$, $N_1 \leq k < N_2$, define
  \[
   W_k = [-J,J] \cap \fnii{f}{[\ep k,\ep(k+1))}.
  \]
Let $N_1 \leq k < N_2$. If $W_k = \varnothing$, set $U_k=V_k=\varnothing$. If $W_k \neq \varnothing$, let $a_k = \inf W_k \geq -J$ and $b_k = \sup W_k \leq J$. For each such $k$ we choose points $c_k,d_k \in F \cap W_k$ and open intervals $U_k$, $V_k$. If $a_k \in F$ then set $c_k=a_k$ and $U_k=\varnothing$. If $a_k \notin F$ then $\mu(\{a_k\})=0$ and $f$ is continuous at $a_k$. From the first fact we glean that there exists an open interval $U_k \ni a_k$ such that $\mu(U_k) < \frac{1}{4(N_2-N_1)}\ep$. From the second fact, together with the density of $F$, we have $\ep k \leq f(a_k)$ and thus we can find a point $c_k \in F \cap W_k \cap U_k \cap (a_k,\infty)$. Likewise, if $b_k \in F$ set $d_k=b_k$ and $V_k=\varnothing$. If $b_k \notin F$, again we can find an open interval $V_k \ni b_k$ such that $\mu(V_k) < \frac{1}{4(N_2-N_1)}\ep$ and a point $d_k \in F \cap W_k \cap V_k \cap (-\infty,b_k)$.
 
 Now define the union of finitely many open intervals
 \[
  U = (-\infty,J) \cup (J,\infty) \cup \bigcup_{k=N_1}^{N_2-1} U_k \cup V_k
 \]
 and let $N \in \N$ large enough so that $c_k,d_k \in M_N$ whenever $N_1 \leq k < N_2$ and $W_k \neq \varnothing$. We estimate
 \[
  \mu(U) \;<\; {\ts \frac{1}{2}\ep + (N_2-N_1)\left(\frac{1}{4(N_2-N_1)}\ep + \frac{1}{4(N_2-N_1)}\ep\right) = \ep.}
 \]
 Now let $y \in \R\setminus U$. Then $y \in [-J,J]$, so $y \in W_k$ for some $k \in \Z$, $N_1 \leq k < N_2$. Then $a_k \leq y \leq b_k$. If $c_k=a_k$ then $c_k \leq y$. If $c_k > a_k$ then because $U_k$ is an interval and $y \notin U_k$, we must have $c_k < y$. In either case, $c_k \leq y$. Likewise $y \leq d_k$. Finally, $c_k,d_k \in W_k$ implies $f(d_k)-f(c_k) < \ep$.
\end{proof}

At last, we are in a position to prove existence.

\begin{theorem}\label{theorem_general_existence}
 The exists a solution $c \in I$ of \eqref{eqn_ODE_system}.
\end{theorem}

\begin{proof}
Apply Proposition \ref{prop_global_particle} to the measures $\mu_n$ supported on $M_n$ to obtain $c_n \in I_{M_n}$ that satisfy \eqref{eqn_ODE_system}. We use these solutions to obtain our element $c \in I$. First of all we define $c(x,\cdot)$ for $x \in F$ before extending the definition to $x \in \R$. Given $x \in F$, we have $x \in M_n$ for all large enough $n$, meaning that $c_n(x,\cdot)$ is defined for all large enough $n$. Since $\Q^+$ and $F$ are countable, we can take a diagonal subsequence of the $c_n$, again labelled $(c_n)$, such that
\[
 c(x,t) \;:=\; \lim_n c_n(x,t)
\]
exists for all $(x,t) \in F \times \Q^+$. According to \eqref{eqn_Lipschitz_time_particle}, we have
\[
 |c_n(x,s) - c_n(x,t)| \leq \pn{K}{\infty}^2\pn{K'''}{\infty}|s-t|,
\]
whenever $s,t \in \R^+$, $x \in M_n$ and $n \in \N$. Hence 
\begin{equation}\label{eqn_Lipschitz_time}
 |c(x,s) - c(x,t)| \leq \pn{K}{\infty}^2\pn{K'''}{\infty}|s-t|,
\end{equation}
whenever $s,t \in \Q^+$ and $x \in F$. Thus for $(x,t) \in F \times \R^+$ we can unambiguously define
\[
 c(x,t) = \lim_n c(t_n,x), 
\]
where $(t_n) \subseteq \Q^+$ is any sequence converging to $t$.

Now we define $c(x,t)$ for $(x,t) \in (\R\setminus F) \times \R^+$. According to Lemma \ref{lem_Phi_est_2}, for all $n \in \N$, $x,y \in M_n$, $x<y$, and $t \in \R^+$ we have
\[
 (y-x)\me^{-At} \leq c_n(y,t) - c_n(x,t) \leq (y-x)\me^{At} + \frac{\Gamma_{n,x,y}}{A}(\me^{At}-1).
\]
Hence by Lemma \ref{lem_Gamma_n} and the above we have
\begin{equation}\label{eqn_building_c_1}
 (y-x)\me^{-At} \leq c(y,t) - c(x,t) \leq (y-x)\me^{At} + \frac{\Gamma_{x,y}}{A}(\me^{At}-1),
\end{equation}
for all $x,y \in F$, $x<y$, and $t \in \R^+$.

Fix $(x,t) \in (\R\setminus F) \times \R^+$ and define
\begin{equation}\label{eqn_definition_c}
 c(x,t) = \lim_n c(x_n,t),
\end{equation}
where $(x_n) \subseteq F$ is any sequence converging to $x$. We show that this definition makes sense. Let $(x_n) \subseteq F$ be a sequence converging to $x$. Given $m,n \in \N$, by \eqref{eqn_building_c_1} we have
\begin{equation}\label{eqn_c_limit}
|c(x_n,t) - c(x_m,t)| \leq |x_n-x_m|\me^{At} + \frac{\Gamma_{x_m,x_n}}{A}(\me^{At}-1).
\end{equation}

By the pseudometric property of $\Gamma$ noted above we have $\Gamma_{x_m,x_n} \leq \Gamma_{x,x_m} + \Gamma_{x,x_n}$. Since $x \notin F$ we know that $\mu(\{x\})=0$, hence by \eqref{eqn_c_limit} and Lemma \ref{lem_accn_points}, the sequence $(c(x_n,t))$ is Cauchy, thus the limit in \eqref{eqn_definition_c} exists. A very similar argument demonstrates that this limit is independent of the choice of sequence $(x_n)$. Therefore \eqref{eqn_definition_c} makes sense as claimed. 

Thus $c$ has been defined on $\Omega$. We observe that the limiting processes used to define $c$ respect \eqref{eqn_Lipschitz_time} and thus this inequality extends to all $s,t \in \R^+$ and $x \in \R$. Next, we show that \eqref{eqn_building_c_1} extends to $x,y \in \R$, $x<y$, and $t \in \R^+$. Let $x,y \in \R$, $x<y$, and let $(x_n),(y_n) \subseteq F$ be sequences converging to $x$ and $y$, respectively. We can and do assume that $x_n < y_n$ for all $n \in \N$. Moreover, if $x \in F$ then we assume that $x_n=x$ for all $n \in \N$, and likewise for $y$. Again by the pseudometric property of $\Gamma$, we have
\[
 |\Gamma_{x_n,y_n} - \Gamma_{x,y}| \leq \Gamma_{x,x_n} + \Gamma_{y,y_n}
\]
for all $n \in \N$, hence by our choice of sequences and Lemma \ref{lem_accn_points}, we obtain $\lim_n \Gamma_{x_n,y_n} = \Gamma_{x,y}$. Thus by applying \eqref{eqn_building_c_1} to $x_n, y_n \in F$ and taking limits, we see that
\begin{equation}\label{eqn_building_c_2}
 (y-x)\me^{-At} \leq c(y,t) - c(x,t) \leq (y-x)\me^{At} + \frac{\Gamma_{x,y}}{A}(\me^{At}-1),
\end{equation}
whenever $x<y$. We glean two important facts from \eqref{eqn_building_c_2}. First, the lower estimate implies that $c(\cdot,t)$ must be strictly increasing. Second, by considering the upper estimate and Lemma \ref{lem_accn_points}, $c(\cdot,t)$ is continuous at $x$ whenever $\mu(\{x\})=0$, and thus all points of discontinuity of $c(\cdot,t)$ (if any) must belong to $P$.

It remains to show that the curves $c(x,\cdot)$, $x \in \R$, are $C^1$ and that $c$ satisfies \eqref{eqn_ODE_system} for $(x,t) \in (0,\infty) \times \R$. We do this in a few steps. First we consider the integral form of \eqref{eqn_ODE_system}. Given $(x,t) \in F \times \R^+$, we have $x \in M_N$ for some $N \in \N$, and from above we know that
\[
 c(x,t) = \lim_{n \geq N} c_n(x,t).
\]
Consider $n \geq N$. As $c_n$ satisfies \eqref{eqn_ODE_system} for $(x,t) \in M_n \times (0,\infty)$, we have
\begin{equation}\label{eqn_integral_n}
c_n(x,t) = x + \lint{0}{t}{H_n(c_n,x,s)^2\lint{z \neq x}{}{K'''(c_n(x,s)-c_n(z,s))}{\mu_n(z)}}{s},
\end{equation}
where
\[
 H_n(c_n,x,s) \;:=\; \lint{-\infty}{\infty}{K(c_n(x,s)-c_n(z,s))}{\mu_n(z)}.
\]
We would like to show that
\begin{equation}\label{eqn_integral}
c(x,t) = x + \lint{0}{t}{H(c,x,s)^2\lint{z \neq x}{}{K'''(c(x,s)-c(z,s))}{\mu(z)}}{s},
\end{equation}
for $(x,t) \in F \times \R^+$. 

To show this, we apply a limiting process to \eqref{eqn_integral_n}. To show that the right hand side of \eqref{eqn_integral_n} converges to that of \eqref{eqn_integral}, we apply Lemmas \ref{lem_integral_conv_1} and \ref{lem_integral_conv_2} below, together with the Dominated Convergence Theorem (due to their technical nature of these results we extract them as separate lemmas).

Next, we must extend \eqref{eqn_integral} to all points $(x,t) \in \Omega$. Let $(x,t) \in (\R\setminus F) \times \R^+$ and let $(x_n) \subseteq F$ converge to $x$. We recall from above that $x$ is a point of continuity of $c(\cdot,t)$. Thus given $z \in \R$ we have
\[
\lim_n K(c(x_n,t)-c(z,t)) = K(c(x,t)-c(z,t)).
\]
By the Dominated Convergence Theorem, we conclude that
\begin{equation}\label{eqn_integral_H}
\lim_n H(c,x_n,t) = H(c,x,t).
\end{equation}

Next, define functions $\map{f,f_n}{\R}{\R}$ by
\[
 f(z) = \begin{cases} K'''(c(x,t)-c(z,t)) &\text{if }z \neq x\\
 0 &\text{if }z = x
 \end{cases}
\]
and
\[
 f_n(z) = \begin{cases} K'''(c(x_n,t)-c(z,t)) &\text{if }z \neq x_n\\
 0 &\text{if }z = x_n.
 \end{cases}
\]
Given $z \neq x$, by the continuity of $c(\cdot,t)$ at $x$ and the continuity of $K'''$ at $c(x,t)-c(z,t) \neq 0$, we have $\lim_n f_n(z) = f(z)$. Since $\mu(\{x\})=0$, it follows that $f_n \to f$ $\mu$-a.e. Again by the Dominated Convergence Theorem,
\begin{align}
\lim_n \lint{z \neq x_n}{}{K'''(c(x_n,t)-c(z,t))}{\mu(z)} \;=& \lim_n \lint{-\infty}{\infty}{f_n(z)}{\mu(z)} \nonumber\\
=& \lint{-\infty}{\infty}{f(z)}{\mu(z)} \nonumber\\
=& \lint{z \neq x}{}{K'''(c(x,t)-c(z,t))}{\mu(z)}. \label{eqn_integral_K'''}
\end{align}
Given that \eqref{eqn_integral} holds for $(x_n,t)$, $n \in \N$, we can apply the Dominated Convergence Theorem, together with \eqref{eqn_integral_H} and \eqref{eqn_integral_K'''}, to conclude that \eqref{eqn_integral} holds for $(x,t)$ as well. 

Finally, we show that \eqref{eqn_integral} implies that $c(x,\cdot)$, $x \in \R$, are $C^1$ and \eqref{eqn_ODE_system} holds for $(x,t) \in M \times (0,\infty)$. The properties will follow from the fact that the integrand
\begin{equation}\label{eqn_integrand}
 s \;\mapsto\; H(c,x,s)^2\lint{z \neq x}{}{K'''(c(x,s)-c(z,s))}{\mu(z)},
\end{equation}
in \eqref{eqn_integral} is continuous. To see this, fix $(x,s) \in \Omega$ and take a sequence $(s_n) \subseteq \R^+$ converging to $s$. Given fixed $z \in \R$, the continuity of $K$, together with \eqref{eqn_H} and \eqref{eqn_Lipschitz_time}, imply
\[
\lim_n K(c(x,s_n) - c(x,s_n)) = K(c(x,s) - c(z,s)).
\]
Hence $\lim_n H(c,x,s_n)=H(c,x,s)$ by the Dominated Convergence Theorem. Similarly, given fixed $z \neq x$, we know that $c(x,s) - c(z,s) \neq 0$, so by \eqref{eqn_Lipschitz_time} and the continuity of $K'''$ on $\R\setminus\{0\}$, we have
\[
 \lim_n K'''(c(x,s_n) - c(x,s_n)) = K'''(c(x,s) - c(z,s)).
\]
By applying the Dominated Convergence Theorem again, it follows that the integrand \eqref{eqn_integrand} is continuous. By the Fundamental Theorem of Calculus, this means that
\begin{equation}\label{eqn_Phi_t}
 c_t(x,t) = H(c,x,t)^2\lint{z \neq x}{}{K'''(c(x,t)-c(z,t))}{\mu(z)},
\end{equation}
and $c_t(x,\cdot)$ is continuous, giving us the two properties we want. This concludes the proof.
\end{proof}

\begin{lemma}\label{lem_integral_conv_1}
Given $(x,t) \in F \times \R^+$, we have
\begin{equation}\label{eqn_integral_conv_1}
 \lim_n H_n(c_n,x,t) = H(c,x,t).
\end{equation}
\end{lemma}

\begin{proof}
 Consider 
 \begin{align}
 H_n(c_n,x,t) - H(c,x,t) \;&=\; \lint{-\infty}{\infty}{K(c_n(x,t)-c_n(z,t))-K(c(x,t)-c(z,t))}{\mu_n(z)} \nonumber\\
 &+\; \lint{-\infty}{\infty}{K(c(x,t)-c(z,t))}{(\mu_n - \mu)(z)}. \label{eqn_integral_conv_2}
\end{align}
We claim that both terms on the right hand side of \eqref{eqn_integral_conv_2} tend to $0$. The convergence of the second term follows by applying Lemma \ref{lem_strong_convergence} to $U=\R$ and the function $\map{f}{\R}{\R}$ defined by $f(z)=K(c(x,t)-c(z,t))$. If $z$ is point of discontinuity of $f$ then it must be a point of discontinuity of $c(\cdot,t)$, whence $z \in P$. 

Now we show that the first term converges to $0$. Let $\ep>0$. Define $f(z)=c(z,t)$, $z \in \R$, and let $N_0 \in \N$ and a union $U$ of finitely many open intervals satisfy the conclusion of Lemma \ref{lem_integral_conv}. By Lemma \ref{lem_strong_convergence} applied to the constant function $\ind{\R}$, let $N_1 \in \N$, $N_1 \geq N_0$, such that $\mu_n(U) < \ep$ whenever $n \geq N_1$. As $x \in F$ and given the pointwise convergence, there exists $N \geq N_1$ such that $x \in M_N$ and $|c_n(y,t)-c(y,t)| < \ep$ whenever $n \geq N$ and $y \in M_{N_0} \cup \{x\}$.

We claim that if $n \geq N$ and $y \in M_n\setminus U$, then $|c_n(y,t)-c(y,t)| < 2\ep$. Given such $y$, we know that there exist $x',z' \in M_{N_0}$ such that $x' \leq y \leq z'$ and $c(z',t)-c(x',t) < \ep$. Together with the fact that $c_n(\cdot,t)$ is increasing and $n \geq N_0$, we can make the estimates
\[
 c_n(y,t) \leq c_n(z',t) \;<\; c(z',t) + \ep \;<\; c(x',t) + 2\ep \leq c(y,t) + 2\ep,
\]
and
\[
c_n(y,t) \geq c_n(x',t) \;>\; c(x',t) - \ep \;>\; c(z',t) - 2\ep \geq c(y,t) - 2\ep,
\]
giving $|c_n(y,t)-c(y,t)| < 2\ep$, as claimed.

It follows that for $n \geq N$
\begin{align*}
& \left|\lint{-\infty}{\infty}{K(c_n(x,t)-c_n(z,t))-K(c(x,t)-c(z,t))}{\mu_n(z)} \right|\\
\leq& 2\pn{K}{\infty}\ep + \lint{z \in M_n \setminus U}{}{|K(c_n(x,t)-c_n(z,t))-K(c(x,t)-c(z,t))|}{\mu_n(z)}\\
\leq& 2\pn{K}{\infty}\ep + \Lip(K)\lint{z \in M_n \setminus U}{}{|c_n(x,t)-c(x,t)| + |c_n(z,t) - c(z,t)|}{\mu_n(z)}\\
\leq& (2\pn{K}{\infty} + 3\Lip(K))\ep.
\end{align*}
Therefore the first term on the right hand side of \eqref{eqn_integral_conv_2} tends to $0$, as required.
\end{proof}

\begin{lemma}\label{lem_integral_conv_2}
 Given $(x,t) \in F \times \R^+$, we have
 \begin{equation}\label{eqn_integral_conv_3}
  \lim_n \lint{z \neq x}{}{K'''(c_n(x,t)-c_n(z,t))}{\mu_n(z)} = \lint{z \neq x}{}{K'''(c(x,t)-c(z,t))}{\mu_n(z)}
\end{equation}
\end{lemma}

\begin{proof}
We reuse much of the argument used in the proof of Lemma \ref{lem_integral_conv_1}; for this reason we just point out the differences. To see that \eqref{eqn_integral_conv_3} holds, consider
\begin{align}
& \lint{z \neq x}{}{K'''(c_n(x,t)-c_n(z,t))}{\mu_n(z)} - \lint{z \neq x}{}{K'''(c(x,t)-c(z,t))}{\mu(z)} \nonumber\\
=& \lint{z \neq x}{}{K'''(c_n(x,t)-c_n(z,t)) - K'''(c(x,t)-c(z,t))}{\mu_n(z)} \label{eqn_integral_conv_5}\\
&+ \lint{z \neq x}{}{K'''(c(x,t)-c(z,t))}{(\mu_n-\mu)(z)}. \label{eqn_integral_conv_6}
\end{align}
The term labelled \eqref{eqn_integral_conv_6} tends to $0$. This follows by applying Lemma \ref{lem_strong_convergence} to $U=(-\infty,x) \cup (x,\infty)$ and the function $\map{f}{\R}{\R}$ defined by $f(z)=K'''(c(x,t)-c(z,t))$, $z \neq x$, and $f(x)=0$. If $z$ is point of discontinuity of $f$ then either $z=x \in F$ or it must be a point of discontinuity of $c(\cdot,t)$, whence $z \in P$. 

The term labelled \eqref{eqn_integral_conv_5} also tends to $0$. Indeed, given $\ep>0$, similarly to above, for large enough $n$ its absolute value is bounded above by $(2\pn{K'''}{\infty} + 3\Lip(K'''))\ep$. This holds because $c_n(\cdot,t)$ and $c(\cdot,t)$ are strictly increasing on $M_n$ and $\R$, respectively, $K'''$ is Lipschitz when restricted to $(-\infty,0)$ or $(0,\infty)$, and the integral can be written as the sum of corresponding integrals on $(-\infty,x)$ and $(x,\infty)$. 
\end{proof}

Thus we have shown that $c$ and the corresponding map $h$ exist. Next we show that $\bar{h}$ has the required regularity properties. For this we will require some more notation.

We will consider the (dual) bounded Lipschitz (or Dudley) norm $\ndot_{BL}$ on $\mathcal{M}(\R)$, defined by
\[
    \n{\mu}_{BL} = \sup\set{\mu(f) := \lint{-\infty}{\infty}{f(x)}{\mu(x)}}{f \in \Xi},
\]
where $\Xi :=\set{f \in C_0(\R)}{\pn{f}{\infty}+\Lip(f) \leq 1}$, together with its associated metric $d$ given by $d(\mu,\nu)=\n{\mu-\nu}_{BL}$, $\mu,\nu \in \mathcal{M}(\R)$. This metric is related to the Wasserstein distance $W_1$ (see e.g.~\cite{bogachev2007}*{Section 8.3}). (Sometimes in the literature $\Xi$ is defined differently, with $\pn{f}{\infty}+\Lip(f)$ replaced by $\max\{\pn{f}{\infty},\Lip(f)\}$, but this change yields an equivalent norm to which the following arguments apply equally.) We will say that a map $\map{m}{\R^+}{\mathcal{M}(\R)}$ is $d$-Lipschitz if there exists a constant $L \in \R^+$ such that
\[
 d(m(s),m(t)) \leq L|s-t|
\]
whenever $s,t \in \R^+$.

\begin{theorem}\label{thm_regularity}
We have $\bar{h}(t) = K*h(t)\in H^3(\R)$ for all $t\in \R^+$. Moreover $\bar{h}\in C_b^{0,\frac{1}{2}}(\R^+;H^3(\R))$.
\end{theorem}

\begin{proof}
First, define a bounded linear map $T:H^3(\R)\rightarrow C_0(\R)$ by
\[
    (Tu)(x) = \langle K(\cdot-x)|u \rangle_{H^3(\R)}.
\]
Since $K\in H^3(\R)$ and is even, by \cite{onaraighpangsmith2022}*{Proposition 2.2}, $T$ is well-defined and the dual operator $\map{T^*}{\mathcal{M}(\R)}{H^3(\R)}$ is given by
\[
    T^*m = K*m.
\]
Thus $\bar{h}(t)\in H^3(\R)$ for all $t\in \R^+$.

Second, we show that the map $\map{h}{\R^+}{\mathcal{M}(\R)}$ is $d$-Lipschitz. Let $f \in \Xi$ and $s,t \in \R^+$. Then
\begin{align}
 |(h(s) - h(t)(f)| =& \left|\lint{-\infty}{\infty}{f(x)}{h(s)(x)} - \lint{-\infty}{\infty}{f(x)}{h(t)(x)} \right| \nonumber\\
 =& \left|\lint{-\infty}{\infty}{f(c(x,s))-f(c(x,t))}{\mu(x)}\right| \nonumber\\
 \leq& \lint{-\infty}{\infty}{|c(x,s))-c(x,t)|}{\mu(x)} \tag*{as $\Lip(f) \leq 1$} \\
 \leq& \lint{-\infty}{\infty}{\pn{K}{\infty}^2\pn{K}{\infty}|s-t|}{\mu(x)} \tag*{by \eqref{eqn_Lipschitz_time}} \nonumber\\
 \leq& \pn{K}{\infty}^2\pn{K}{\infty}|s-t|. \label{eqn_h_d-Lipschitz}
\end{align}
(As observed in the proof of Theorem \ref{theorem_general_existence}, \eqref{eqn_Lipschitz_time} applies whenever $x \in \R$ and $s,t \in \R^+$.) Consequently,
\[
 d(h(s),h(t)) = \sup\set{|(h(s) - h(t)(f)|}{f \in \Xi} \leq \pn{K}{\infty}^2\pn{K}{\infty}|s-t|.
\]
Finally, since $K\in W^{3,1}(\R)$, $K'''$ has bounded variation and $h$ is $d$-Lipschitz, by \cite{onaraighpangsmith2022}*{Proposition 4.4}, we have $\bar{h}\in C_b^{0,\frac{1}{2}}(\R^+;H^3(\R))$. 
\end{proof}

\section{Uniqueness}\label{sect_uniqueness}

In this section we show that the solutions $c \in I$ obtained above are unique. We will obtain a local uniqueness result before following with a global one. The strategy will be to use maps between metric spaces that have contraction-like properties. To this end we introduce some more notation. Recall the spaces $X_{M,T}$ and $I_{M,T}$ from Section \ref{sec:weak} and set $X_T=X_{\R,T}$ and $I_T=X_{\R,T}$. By compactness and local boundedness, we can equip $X_T$ with the seminorms
\[
\pn{c}{n} \;:=\; \sup\set{|c(x,t)|}{(x,t) \in [-n,n] \times [0,T]}, \qquad n \in \N.
\]
(These seminorms endow $X_T$ with a Fr\'echet space structure.) Also fix a strictly increasing map $\map{c_0}{\R}{\R}$ and define $\map{\Phi,\Phi_n}{I_T}{X_T}$, $n \in \N$, by
\begin{equation}\label{eqn_Phi_2}
 (\Phi c)(x,t) = c_0(x) + \lint{0}{t}{H(c,x,s)^2 \lint{z \neq x}{}{K'''(c(x,s)-c(z,s))}{\mu(z)}}{s}
\end{equation}
and
\begin{equation}\label{eqn_Phi_3}
 (\Phi_n c)(x,t) = c_0(x) + \lint{0}{t}{H_n(c,x,s)^2 \lint{\substack{|z| \leq n \\ z \neq x}}{}{K'''(c(x,s)-c(z,s))}{\mu(z)}}{s},
\end{equation}
where
\begin{equation}\label{eqn_H_truncated}
 H_n(c,x,t) \;:=\; \lint{|z|\leq n}{}{K(c(x,t)-c(z,t))}{\mu(z)}.
\end{equation}

We continue with two lemmas.

\begin{lemma}\label{lem_cont}
Given $c,c' \in I_T$, we have $\pn{\Phi_n c - \Phi_n c'}{n} \leq AT\pn{c-c'}{n}$.
\end{lemma}

\begin{proof}
Given $(x,t) \in [-n,n] \times [0,T]$, we calculate
\begin{align}
 & (\Phi_n c)(x,t) - (\Phi_n c')(x,t) \nonumber\\
 =& \lint{0}{t}{\big(H_n(c,x,s)^2 - H_n(c',x,s)^2\big)\lint{\substack{|z| \leq n \\ z \neq x}}{}{K'''(c(x,s)-c(z,s))}{\mu(z)}}{s} \nonumber\\
 &+ \lint{0}{t}{H_n(c',x,s)^2\lint{\substack{|z| \leq n \\ z \neq x}}{}{K'''(c(x,s)-c(z,s))-K'''(c'(x,s)-c'(z,s))}{\mu(z)}}{s}. \label{eqn_cont_1}
\end{align}
Using \eqref{eqn_H_truncated} and the fact that $\pn{\mu}{1} \leq 1$, we see that
\[
|H_n(c,x,t)-H_n(c',x,t)| \leq 2\Lip(K)\pn{c-c'}{n}.
\]
Because $c,c' \in I_T$ and $K'''\restrict{(0,\infty)}$ is Lipschitz, we can make the estimate
\begin{align*}
& \left|\lint{\substack{|z| \leq n \\ z < x}}{}{K'''(c(x,s)-c(z,s))-K'''(c'(x,s)-c'(z,s))}{\mu(z)}\right|\\
\leq & \Lip(K''')\lint{\substack{|z| \leq n \\ z < x}}{}{|c(x,s)-c'(x,s)| + |c(z,s)-c'(z,s)|}{\mu(z)}\\
\leq & 2\Lip(K''')\pn{c-c'}{n}\lint{\substack{|z| \leq n \\ z < x}}{}{}{\mu(z)}.
\end{align*}
Similarly, as $K'''\restrict{(-\infty,0)}$ is Lipschitz
\begin{align*}
& \left|\lint{\substack{|z| \leq n \\ x < z}}{}{K'''(c(x,s)-c(z,s))-K'''(c'(x,s)-c'(z,s))}{\mu(z)}\right|\\
\leq & 2\Lip(K''')\pn{c-c'}{n}\lint{\substack{|z| \leq n \\ x < z}}{}{}{\mu(z)}.
\end{align*}
Applying the last three estimates to \eqref{eqn_cont_1} and using again the fact that $\pn{\mu}{1} \leq 1$, we have
\begin{align*}
 & |(\Phi_n c)(x,t) - (\Phi_n c')(x,t)| \\
 \leq\;& 4\pn{K}{\infty}\Lip(K)\pn{K'''}{\infty}\pn{c-c'}{n}t \;+\; 2\pn{K}{\infty}^2\Lip(K''')\pn{c-c'}{n}t \\
 \leq\;& AT\pn{c-c'}{n}. \qedhere
\end{align*}
\end{proof}

\begin{lemma}\label{lem_cont_2}
Given $c \in I_T$, we have $\lim_n \pn{\Phi_n c - \Phi c}{n}=0$.
\end{lemma}

\begin{proof}
Let $(x,t) \in [-n,n] \times [0,T]$. Then we have
\begin{align}
 & (\Phi_n c)(x,t) - (\Phi c)(x,t) \nonumber\\
 =& \lint{0}{t}{\big(H_n(c,x,s)^2 - H(c,x,s)^2\big)\lint{\substack{|z| \leq n \\ z \neq x}}{}{K'''(c(x,s)-c(z,s))}{\mu(z)}}{s} \nonumber\\
 -& \lint{0}{t}{H(c,x,s)^2\lint{\substack{|z| > n \\ z \neq x}}{}{K'''(c(x,s)-c(z,s))}{\mu(z)}}{s}. \label{eqn_cont_2_1}
\end{align}
By \eqref{eqn_H} and \eqref{eqn_H_truncated}, we have
\[
|H_n(c,x,t)-H(c,x,t)| \leq \pn{K}{\infty}\mu(\R\setminus[-n,n]).
\]
Using this and \eqref{eqn_cont_2_1} we can estimate
\[
|(\Phi_n c)(x,t) - (\Phi c)(x,t)| \leq 3T\pn{K}{\infty}^2\pn{K'''}{\infty}\mu(\R\setminus[-n,n]).
\]
from which the result follows.
\end{proof}

Now we can prove a local uniqueness result. 

\begin{proposition}\label{prop_local_uniqueness}
Suppose that $c,c' \in I_{\frac{1}{A}}$ have the property that $c(x,0)=c_0(x) = c'(x,0)$, $x \in M$, and satisfy \eqref{eqn_ODE_system} for $(x,t) \in \R \times (0,\frac{1}{A})$. Then $c=c'$.
\end{proposition}

\begin{proof}
Given such $c$ and $c'$ we observe that $\Phi c = c$ and $\Phi c' = c'$. By continuity of the maps $c(x,\cdot)$, $x \in M$, it is sufficient to show that $c(x,t) = c'(x,t)$ for all $(x,t) \in \R \times [0,\frac{1}{A})$. For a contradiction assume that $c(x,t) \neq c'(x,t)$ for some $(x,t) \in \R \times [0,\frac{1}{A})$. Pick $T \in (\frac{1}{A},t)$ and hereafter consider the restrictions of $c$ and $c'$ to $I_T$ (still labelled $c$ and $c'$, respectively).

Given $n \geq |x|$, we have $\pn{c-c'}{n} \geq |c(x,t) - c'(x,t)| > 0$. By Lemma \ref{lem_cont}
\[
\pn{\Phi_n c-\Phi_n c'}{n} \leq AT\pn{c-c'}{n}.
\] 
By the triangle and reverse triangle inequalities
\begin{align*}
|\pn{c-c'}{n} - \pn{\Phi_n c - \Phi_n c'}{n}| \leq& \pn{(c-\Phi_n c) - (c'-\Phi_n c')}{n}\\
\leq& \pn{(c-\Phi_n c)}{n} + \pn{(c'-\Phi_n c')}{n}.
\end{align*}
Hence by Lemma \ref{lem_cont_2}, there exists $N \in \N$, $N \geq |x|$, such that
\[
 |\pn{c-c'}{n} - \pn{\Phi_n c - \Phi_n c'}{n}| \;<\; (1-AT)|c(x,t) - c'(x,t)|,
\]
whenever $n \geq N$. Therefore, for such $n$,
\[
\pn{c-c'}{n} \geq AT\pn{c-c'}{n} + (1-AT)|c(x,t) - c'(x,t)| \;>\; \pn{c-c'}{n}.
\]
This contradiction ensures that $c=c'$.
\end{proof}

With this result we can prove global uniqueness.

\begin{theorem}\label{theorem_global_uniqueness}
Suppose that $c,c' \in I$ satisfy the hypotheses of Theorem \ref{theorem_ODE}. Then $c=c'$.
\end{theorem}

\begin{proof} For a contradiction, we suppose otherwise that $c \neq c'$. Then
\[
T_1 \;:=\; \sup\set{t \in \R^+}{c(x,s) = c'(x,s) \text{ whenever }(x,s) \in [0,t] \times M} \;<\; \infty.
\]
By Proposition \ref{prop_local_uniqueness}, $T_1 \geq \frac{1}{A}$. Let $T_0 = T_1-\frac{1}{2A}$. Define $d,d' \in I_{\frac{1}{A}}$ by
\[
d(x,t) = c(x,T_0+t) \;\text{and}\; d'(x,t) = c'(x,T_0+t), \quad (x,t) \in \R \times [0,{\ts\frac{1}{A}}].
\]
Using \eqref{eqn_H}, observe that
\begin{align*}
H(d,x,t) \;&=\; \lint{-\infty}{\infty}{K(d(x,t)-d(z,t))}{\mu(z)} \\
&=\; \lint{-\infty}{\infty}{K(c(x,T_0+t)-c(z,T_0+t))}{\mu(z)} = H(c,x,T_0+t).
\end{align*}
Hence for $(x,t) \in \R \times (0,\frac{1}{A})$ we have
\begin{align*}
d_t(x,t) \;&=\; c_t(x,T_0+t)\\
&=\; H(c,x,T_0+t)^2\lint{z \neq x}{}{K'''(c(x,T_0+t)-c(z,T_0+t))}{\mu(z)}\\
&=\; H(d,x,t)^2\lint{z \neq x}{}{K'''(d(x,t)-d(z,t))}{\mu(z)},
\end{align*}
so $d$ satisfies \eqref{eqn_ODE_system} for $(x,t) \in \R \times (0,\frac{1}{A})$. The same applies to $d'$. Since $d(\cdot,0)=c(\cdot,T_0)=c'(\cdot,T_0)=d'(\cdot,0)$ is strictly increasing, $d=d'$ by Proposition \ref{prop_local_uniqueness}. However, this implies that $c(x,t) = c'(x,t)$ for all $(x,t) \in \R \times [T_0,T_0+\frac{1}{A}]$. Since $T_0+\frac{1}{A} > T_1$, we get a contradiction. 
\end{proof}

\section{Final remarks}\label{sect_final_remarks}

We begin this section with a remark about continuity.

\begin{remark}\label{rem_discontinuity}
We already know from the proof of Theorem \ref{theorem_general_existence} that if $x \notin P$ then $x$ is a point of continuity of $c(\cdot,t)$ for all $t \in \R^+$. Conversely, we see straightaway from \eqref{eqn_c_lower_est} that
\begin{equation}\label{eqn_lower_estimate_refined}
c(y,t)-c(x,t) \geq \frac{\Lambda_{x,y}}{A}(1-\me^{-At})
\end{equation}
whenever $x < y$ and $t \in (0,\infty)$. Thus
\[
 c(x,t) + \frac{\pn{K}{\infty}^2\pn{K'''}{\infty}\mu(\{x\})^3}{A}(1-\me^{-At}) \leq \lim_{y \to x_+} c(y,t),
\]
and 
\[
 \lim_{y \to x_+} c(y,t) \leq c(x,t) - \frac{\pn{K}{\infty}^2\pn{K'''}{\infty}\mu(\{x\})^3}{A}(1-\me^{-At})
\]
whenever $t > 0$. Thus we have quantitative estimates which show that if $x \in P$ then $x$ is not a point of left or right-continuity of $c(\cdot,t)$, for any $t>0$. 
\end{remark}

Next, we show that the class of solutions of \eqref{eq:weak} is wider than the one hitherto established via push-forward measures as described above. In so doing, we show that, with respect to this wider class, we do not have uniqueness of solutions. In our opinion this provides justification for pursuing the (longer and more technical) push-forward approach.

The starting point for obtaining this wider class is the following metric Arzel\`a-Ascoli theorem, which is a simplification of \cite{ambrosio2005}*{Proposition 3.3.1}.

\begin{theorem}[cf. \cite{ambrosio2005}*{Proposition 3.3.1}]\label{thm:arzela}
Let $(X,\tau)$ be a se\-quen\-tially compact Hausdorff topological space, and let $d$ be a $\tau$-lower semi\-con\-ti\-nuous metric on $X$. Let $f_n:\R^+\rightarrow X$ such that
\begin{equation}
    \limsup_{n\rightarrow\infty} d(f_n(s),f_n(t)) \leq L|s-t| \qquad \forall\,s,t\in\R^+,
\end{equation}
for some $L\geq0$. Then there exists a subsequence of $(f_n)$, labelled in the same way, and $f:\R^+\rightarrow X$ such that
\begin{equation}
    f_n(t) \xrightarrow{\tau} f(t) \qquad \forall\,t\in\R^+,
\end{equation}
and $f$ is $d$-Lipschitz with Lipschitz constant $L$.
\end{theorem}

For our purpose, we set $(X,\tau)=(B_{\mathcal{M}^+(\R)},w^*)$, where
\[
B_{\mathcal{M}^+(\R)}:=\set{\mu\in\mathcal{M}^+(\R)}{\pn{\mu}{1} \leq 1}.
\]
As $C_0(\R)$ is separable, $(X,\tau)$ is metrizable and compact, hence sequentially compact.

Next, given an initial non-zero measure $\mu \in B_{\mathcal{M}^+(\R)}$, let $(\mu_n) \subseteq B_{\mathcal{M}^+(\R)}$ denote a sequence of non-zero, finitely supported measures that converges to $\mu$ in the $w^*$-topology (we do not need to carefully specify the $\mu_n$ as we had to in Section \ref{sect_existence}; indeed, the greater freedom here will lead to non-uniqueness of solutions). By applying Proposition \ref{prop_global_particle}, we obtain functions $\map{h_n}{\R^+}{B_{\mathcal{M}^+(\R)}}$ which, given Theorem \ref{theorem_ODE}, will satisfy \eqref{eq:weak} with initial condition $h_n(0)=\mu_n$.

By \eqref{eqn_h_d-Lipschitz}, we have
\[
 d(h_n(s),h_n(t)) \leq \pn{K}{\infty}^2\pn{K}{\infty}|s-t|,
\]
independently of $n$. By Theorem \ref{thm:arzela}, there exists a subsequence of $(h_n)$, labelled in the same way, and $\map{h}{\R^+}{B_{\mathcal{M}^+(\R)}}$, such that 
\[
    h_n(t) \xrightarrow{w^*} h(t) \qquad \text{for all }t \in \R.
\]
Establishing that $h$ is a weak solution of \eqref{eq:hdefall} uses a variant of the argument in \cite{onaraighpangsmith2022}*{Proposition 5.2}.

\begin{theorem}\label{thm:conv}
$(h$, $\bar{h})$ defined above is a weak solution of \eqref{eq:hdefall}. 
\end{theorem}

\begin{proof}
We show that $(h,\bar{h})$ satisfy \eqref{eq:weak}. The first term follows directly from the fact that $h_n(0)\to \mu$ in $(B_{\mathcal{M}^+(\R)},w^*)$ and $\phi(\cdot,0)$ is in the predual space $C_0(\R)$. From \cite{onaraighpangsmith2022}*{Proposition 5.2}, we have that $\partial_x^k\bar{h}_n\to\partial_x^k\bar{h}$ in the topology of $L^1_{loc}(\Omega)$, for $0 \leq k \leq 3$, so the linear term converges as $n\to\infty$. Furthermore, since for all $0 \leq k \leq 3$, the functions $(\partial_x^k\bar{h}_n),\partial_x^k\bar{h}\in L^\infty(\Omega)$ are uniformly bounded with respect to $\|\cdot\|_\infty$, the nonlinear terms in \eqref{eq:weak} also converge as $n\to\infty$. 
\end{proof}

Finally we give an example to show that, with respect to this wider class, we fail to have uniqueness of solutions.

\begin{example}\label{ex_non-uniqueness}
It is evident by inspection that the function $c \in I_{\{0\}}$, given by $c(0,t)=0$, $t \in \R^+$, is a solution of \eqref{eqn_ODE_system} (equivalently \eqref{eq:ivp}), with respect to the Dirac measure $\delta_0$, and thus the corresponding function $h_0$ given by $h_0(t)=\delta_0$, $t \in \R^+$, is a solution of \eqref{eq:weak} with initial condition $h_0(0)=\delta_0$.

We will obtain another solution $h$ of \eqref{eq:weak} that satisfies $h(0)=\delta_0$. Given $n \in \N$, set $\mu_n = \frac{1}{2}(\delta_{-n^{-1}} + \delta_{n^{-1}})$. By Proposition \ref{prop_global_particle}, there exists a solution $c_n \in I_{\{-n^{-1},n^{-1}\}}$ of \eqref{eqn_ODE_system} (equivalently \eqref{eq:ivp}), with corresponding function $h_n$ given by $h_n(t) = \frac{1}{2}(\delta_{c(-n^{-1},t)} + \delta_{c(n^{-1},t)})$ that solves \eqref{eq:weak} with initial condition $h_n(0)=\mu_n$.

Because $\mu_n \stackrel{w^*}{\to} \delta_0$, by taking as subsequence as necessary as above, there exists a solution $h$ of \eqref{eq:weak} such that $h_n(t) \xrightarrow{w^*} h(t)$ for all $t \in \R^+$ and having initial condition $h(0)=\delta_0$.

Now we show that $h \neq h_0$. By appealing to uniqueness of the particle solutions, using the facts that $K$ and $K'''$ are even and odd functions, respectively, and the $\mu_n$ are symmetric about the origin, we have $c_n'(-n^{-1},t)=-c_n'(n^{-1},t)$ for all $t>0$. Then, using \eqref{eqn_lower_estimate_refined} we have
\[
 2c_n'(n^{-1},t) = c_n'(n^{-1},t) - c_n'(-n^{-1},t) \geq \frac{\pn{K}{\infty}^2 \pn{K'''}{\infty}}{8A}(1-\me^{-At}).
\]
As $h_n(t) \stackrel{w^*}{\to} h(t)$, and observing the symmetry, we conclude that there is a function $\map{\beta}{\R^+}{\R^+}$ such that $h(t)=\frac{1}{2}(\delta_{-\beta(t)}+\delta_{\beta(t)})$, $t \in \R^+$, where
\[
 \beta(t) \geq \frac{\pn{K}{\infty}^2 \pn{K'''}{\infty}}{16A}(1-\me^{-At}), \qquad t \in \R^+.
\]
In particular, $h(t) \neq h(0)$ for all $t>0$.
\end{example}

In the example above, the function $h_0$ follows a single path while $h$ immediately splits into two paths. The wider class of weak solutions admits all kinds of path-splitting. This kind of splitting is prohibited in the push-forward case, to the extent that solutions are rendered unique. We remark that, while it is helpful for mathematical completeness to consider all measures in $B_{\mathcal{M}^+(\R)}$ as potential initial conditions, initial measures $\mu$ having atoms (i.e. where $P \neq \varnothing$) are not physically relevant. This observation prompts the following question.

\begin{problem}
 If the initial measure $\mu$ is diffuse, i.e.~non-atomic, is there only one solution within the class of all weak solutions that is unique? 
\end{problem}

\subsection*{Acknowledgements}

 This publication has emanated from research supported in part by a Grant from Science Foundation Ireland under Grant Number 18/CRT/6049. LON has also been supported by the ThermaSMART network. The Ther\-ma\-SMART network has received funding from the European Union’s Horizon 2020 re\-search and innovation programme under the Marie Sklodow\-ska–Curie grant agree\-ment No. 778104.

\bibliography{reference}

\end{document}